\documentclass[12pt]{amsart}
\usepackage{graphicx,amssymb,amsfonts,latexsym,amsmath,amsthm,times, amscd, euscript, MnSymbol}
\makeatletter
\def\input@path{{figures/}}
\makeatother
\usepackage[utf8x]{inputenc}
\usepackage{pst-all}

\usepackage{color}
\usepackage[english]{babel}

\numberwithin{equation}{section}
\usepackage[active]{srcltx}

\usepackage{enumerate}
\usepackage{verbatim}
\usepackage{multicol}

\newtheorem{theorem}{\bf Theorem}[section]
\newtheorem{lemma}[theorem]{\bf Lemma}
\newtheorem{corollary}[theorem]{\bf Corollary}
\newtheorem{proposition}[theorem]{\bf Proposition}

\newtheorem{observation}{\noindent {\bf Observation}}
\newtheorem{note}{\bf Note}[section]
\newtheorem{claim}{\noindent {\bf Claim}}
\newtheorem{case}{\noindent {\bf Case}}

\newtheorem{fact}{\noindent{\bf Fact}}

\newcommand{\NN}{{\mathbb N}}

\DeclareMathOperator{\age}{Age}

\newtheorem{remark}[theorem]{Remark}

\def\endproof{\hfill {\kern 6pt\penalty 500
\raise -0pt\hbox{\vrule \vbox to5pt {\hrule width 5pt
\vfill\hrule}\vrule}}}

\newcommand{\qedclaim}{\hfill $\diamond$}
\newenvironment{proofclaim}{\noindent{\bf \emph{Proof.}}}{\qedclaim\medbreak}

\title[Hereditary classes of ordered binary structures]{Hereditary classes of ordered binary structures}
\author[D.Oudrar] {Djamila Oudrar}
\thanks{*The author was supported by CMEP-Tassili grant}
\address{Faculty of Mathematics, USTHB, Algiers, Algeria}
\email {dabchiche@usthb.dz; djoudrar@gmail.com}

\date{\today}

\begin{document}

\begin{abstract}
Balogh, Bollob\'{a}s and Morris (2006) have described a threshold phenomenon in the behavior of the profile of hereditary classes of ordered graphs. In this paper, we give an other look at their result based on the notion of monomorphic decomposition of a relational structure introduced in \cite{P-T-2013}. We prove that the class $\mathfrak S$ of ordered binary structures which do not have a finite monomorphic decomposition has a finite basis (a subset $\mathfrak A$ such that every member of $\mathfrak S$ embeds some member of $\mathfrak A$). In the case of ordered reflexive directed graphs, the basis has 1242 members and the profile of their ages grows at least as the Fibonacci function. From this result, we deduce that the following dichotomy property holds for every hereditary class $\mathfrak C$ of finite ordered binary structures of a given finite type. Either there is an integer $\ell$ such that every member of $\mathfrak C$ has a monomorphic decomposition into at most $\ell$ blocks and in this case the profile of $\mathfrak C$ is bounded by a polynomial of degree $\ell-1$ (and in fact is a polynomial), or $\mathfrak C$ contains the age of a structure which does not have a finite monomorphic decomposition, in which case the profile of  $\mathfrak C$ is bounded below by the Fibonacci function.
\end{abstract}

\maketitle

 \noindent {\bf AMS Subject Classification:} 05C30, 06F99, 05A05, 03C13.

\vspace{.08in} \noindent \textbf{Keywords}: profile, monomorphic decomposition, ordered graphs, ordered binary relational structures, ordered set, well quasi-ordering.

\section{Introduction and presentation of the results}
The framework of this paper is the theory of relations. It is about a counting function, the profile. The \emph{profile} of a class  $\mathfrak {C}$  of finite relational structures (also called the \emph{speed} by other authors)  is the integer function  $\varphi _{\mathfrak{C}}$  which counts for each non negative integer $n$  the number $\varphi _{\mathfrak{C}}(n)$ of members of   $\mathfrak{C}$ defined
on $n$ elements, isomorphic structures being  identified.  For the last fifteen years, the behavior of this function has been discussed in many papers, particularly
when  $\mathfrak{C}$ is
\emph{hereditary} (that is contains  every substructure of any member of  $\mathfrak{C}$) and is made of graphs (directed or not), of tournaments, of ordered sets, of ordered  graphs and of ordered hypergraphs. As observed by P. J. Cameron \cite{cameron}, numerous results obtained about classes of permutations obtained during the same period fall under the frame of the profile of hereditary classes of relational
structures, namely \emph{bichains}, that is
structures made of two linear orders on the same set. The results show that the profile cannot be arbitrary: there are \emph{jumps} in its possible growth rates. Typically, its growth  is polynomial or faster than every polynomial (\cite{pouzet.tr.1978} for ages, see  \cite{pouzet06} for a survey) and for several classes of structures, it is either at least exponential (e.g. for tournaments \cite{BBM07,Bou-Pouz}, ordered graphs and hypergraphs \cite{B-B-M06,B-B-M06/2,klazar1} and permutations \cite{kaiser-klazar})
or  at least with the growth of the partition function (e.g. for graphs \cite {B-B-S-S}). For more, see  the survey of Klazar \cite{klazar} and for permutations the survey of Vatter \cite{vatter4}.

This paper is motivated by Balogh, Bollob\'{a}s and Morris results about the profile of hereditary classes of ordered graphs.
They show in  \cite{B-B-M06/2} that if $\mathfrak C$ is a hereditary class of finite ordered graphs then its profile $\varphi_{\mathfrak C}$ is either polynomial  or is ranked by the Fibonacci functions.

Their theorem states:
\begin{theorem}\label{thm:BBM}
If $\mathfrak C$ is a hereditary class of finite  ordered graphs, then one of the following assertions holds.
\begin{enumerate}
\item[(a)] $\varphi_{\mathfrak C}(n)$ is bounded above and there exist $M, N\in\NN$ such that $\varphi_{\mathfrak C}(n)=M$ for every $n\geq N$.
\item[(b)] $\varphi_{\mathfrak C}(n)$ is a polynomial in $n$. There exist $k\in\NN$ and integers $a_0,\dots,a_k$ such that, $\varphi_{\mathfrak C}(n)=\underset{i=0}{\overset{k}{\sum}}a_i\binom{n}{i}$ for all sufficiently large $n$, and $\varphi_{\mathfrak C}(n)\geq n$ for every $n\in\NN$.
    \item[(c)] $F_{n,k}\leq \varphi_{\mathfrak C}(n)\leq p(n)F_{n,k}$ for every $n\in\NN$, for some $2\leq k\in\NN$ and some polynomial $p$, so in particular $\varphi_{\mathfrak C}(n)$ is exponential.
\item[(d)] $\varphi_{\mathfrak C}(n)\geq 2^{n-1}$ for every $n\in\NN$.
\end{enumerate}
\end{theorem}

Here,   $F_{n,k}$ denote the $n^{\text{th}}$ generalized Fibonacci number of order $k$, defined by $F_{n,k}=0$ if $n<0$, $F_{0,k}=1$ and $F_{n-1,k}=F_{n-2,k}+F_{n-3,k}+\dots+F_{n-k,k}$ for every $n\geq 1$. For $k=2$, $F_{n,2}$ is the Fibonacci number $F_n$.

Their result extends Kaiser and Klazar result for classes of permutations (see \cite{kaiser-klazar}).

\smallskip

In this paper, we give an other look at Balogh, Bollob\'{a}s and Morris result using notions of the theory of relations and notably the notion of monomorphic decomposition of a relational structure (described in Pouzet and Thi\'ery 2013 \cite{P-T-2013}).

Our technique allows to characterize the hereditary classes of ordered directed graphs and more generally ordered binary structures which have a polynomially bounded profile. It gives the jump of the profile between polynomials and the ordinary  Fibonacci function $F_{-,2}$ but does not gives the hierarchy given in $(c)$ and $(d)$ of Theorem \ref{thm:BBM}.

We recall that a \emph{monomorphic decomposition} of a relational structure $\mathcal R$ defined on a set $V$ is a partition $(V_i)_{i\in I}$ of $V$ such that the induced structures $\mathcal R_{\restriction A}$ and $\mathcal R_{\restriction A'}$ on two finite subsets $A,A'$ of $V$ are isomorphic provided that the sets $A\cap V_i$ and $A'\cap V_i$ have the same cardinality for each $i\in I$. We also recall that the \emph {age} of a relational structure $\mathcal R$ is the set $\age(\mathcal R)$ of finite induced substructures of $\mathcal R$ considered up to isomorphism. We will call \emph{profile of } $\mathcal R$, denoted by $\varphi_{\mathcal R}$,  the profile of its age $\age (\mathcal R)$.

We prove first a dichotomy  result about this notion:
\begin{theorem}\label{thm:dichotomy-binary structure}
Let $\mathfrak C$ be a hereditary class of finite relational structures with a fixed finite signature. Then
\begin{enumerate}
\item\label{itemtheo1} Either  $\mathfrak C$ is a finite union of ages of relational structures, each admitting a finite monomorphic decomposition. \item\label{itemtheo2} Or $\mathfrak C$ contains the age $\mathfrak D$ of a structure $\mathcal R$ that does not have a finite monomorphic decomposition, this age being minimal with this property.
\end{enumerate}
\end{theorem}

From this follows easily a characterization of classes satisfying the first item of  Theorem \ref{thm:dichotomy-binary structure}.
\begin{corollary}\label{cor:first item}
A class $\mathfrak C$ is a finite union of ages of relational structures that have a finite monomorphic decomposition  if and only if there is an integer $\ell$ such that every member of $\mathfrak C$ has a monomorphic decomposition into at most $\ell$ blocks.
\end{corollary}

Clearly, the profile $\varphi_{\mathcal R}$  of a structure $\mathcal R$ admitting a monomorphic decomposition is bounded above by a polynomial (in fact, $\varphi_{\mathcal R}(n)\leq \binom{n+\ell-1}{\ell-1}$,  where $\ell$ is the number of blocks of the monomorphic decomposition). Hence, if $\mathfrak C$ satisfies  (\ref{itemtheo1}) of Theorem \ref{thm:dichotomy-binary structure},  the profile of $\mathfrak C$ is bounded by a polynomial. It was shown in \cite{P-T-2013} that $\varphi_{\mathcal R }$, hence $\varphi_{\mathfrak C}$,  is a quasi-polynomial (that is a polynomial $a_{\ell-1}(n)n^{\ell-1}$ $+\cdots$ $+a_{1}(n)n$ $+a_{0}(n)$ whose coefficients $a_{\ell-1}(n)$, $\cdots$, $a_{0}(n)$ are periodic functions).
In the case of ordered structures, this profile is in fact polynomial (see \cite{oudrar-pouzet}).

The profile of a class verifying (\ref{itemtheo2}) of Theorem \ref{thm:dichotomy-binary structure} is bounded below by the profile of the age $\mathfrak D$. For arbitrary relational structures, this profile can be bounded above by a polynomial. But, in the case of ordered structures, it is necessarily at least exponential \cite{oudrar-pouzet} and, as we will see in Proposition \ref{prop:profil-ordonne}, in the case of ordered binary structures, the profile of $\mathfrak D$ is bounded below by the Fibonacci function.

Consequently:

\begin{theorem}
If $\mathfrak C$ is a hereditary class of ordered binary structures then
\begin{enumerate}
\item Either the profile is bounded above by a polynomial, and in this case there is an integer $\ell$ such that every member of $\mathfrak C$ has a monomorphic decomposition in at most $\ell$ blocks.
    \item Or the profile is bounded below by the Fibonacci function.
\end{enumerate}
\end{theorem}

\medskip

Ordered  structures that have a finite monomorphic decomposition have a particularly simple form. If $\mathcal R: =(V, \leq, \rho_1, \dots, \rho_k)$ is such a structure,  the chain  $(V, \leq)$ decomposes into finitely many intervals $V_i$ such that the union of any local isomorphisms $f_i$ of $(V_i, \leq_{\restriction V_i})$ is a local isomorphism of $\mathcal R$ (see Theorem \ref{thm:charcfinitemonomorph} below). If $\mathcal R$ is made of binary relations, these relations are quite close to the given order. For example, if $\mathcal R$ is a \emph{bichain}, that is  $\mathcal R:= (V, \leq, \leq')$,  where $\leq'$ is a linear order, then $\leq'$ coincides with $\leq$ or its opposite on each $V_i$. Note that any ordered binary structure  can be viewed as superposition of graphs (symmetric and irreflexive) and unary relations on the same ordered set and this superposition has the same local isomorphisms as $\mathcal R$ hence the same profile.
Indeed, if $\mathcal R:=(V,\leq,\rho_1,\dots,\rho_k)$ is an ordered binary structure, replace each $\rho_i$ by the graphs $\rho_i^+$, $\rho_i^-$ and the unary relation $u(\rho_i)$ defined by setting $\rho_i^+:=\{\{x,y\}\in [V]^2:x<y \text{ and }(x,y)\in\rho_i\}$, $\rho_i^-:=\{\{x,y\}\in [V]^2:x<y \text{ and }(y,x)\in\rho_i\}$ and $u(\rho_i):=\{x\in V:(x,x)\in\rho_i\}$. For example, if $\mathcal R$ has a finite monomorphic decomposition $(V_i)_{i:=1, \dots, \ell}$ then the graphs are full or empty on each $V_i$.

The following result indicates that for  ordered binary structures, the study of their profile reduces (roughly) to the case of  single ordered binary relations, and in fact to  unary ordered structures  or to    ordered graphs.

\begin{proposition}\label{prop:decomp-finie}
Let $k\in \mathbb N$. An ordered binary structure $\mathcal R:= (V, \leq, \rho_1,\dots,\rho_k)$ has a finite monomorphic decomposition if and only if every structure $\mathcal R_i:=(V, \leq, \rho_i)$,  $(1\leq i\leq k)$,  has such a decomposition.
\end{proposition}

The case of an ordered unary relation is handled by a result of P. Jullien \cite{jullien}: the profile is either polynomial or bounded below by the exponential function $n\hookrightarrow 2^n$.  The case of ordered graphs was handled by Balogh, Bollob\'{a}s and Morris \cite{B-B-M06/2}.

Finite ordered structures, of the form $\mathcal R:=(V,\leq,u_1,\dots,u_k)$, consisting of a linear order $\leq$ and unary relations $u_i$, can be represented by words over a finite alphabet $A$ (namely the set $\{0,1\}^k$). Embedding between structures correspond to the subword ordering. A famous result of Higman \cite{higman52} asserts that the set $A^{\ast}$ of words over a finite alphabet $A$ is well-quasi-ordered for this ordering. Hence, hereditary classes of words are finite unions of ideals. According to \cite{jullien} (see Chapter 6, page 103), each ideal decomposes into a finite product of \emph{elementary ideals},   sets  of the form $\{\square,a\}$, where $\square$ is the empty word and $a\in A$  and of \emph{starred ages}, sets of the form $B^{\ast}$ where $B^{\ast}$ is the set of words over $B\subseteq A$ (cf. \cite {kabil-pouzet} for an extension to an ordered alphabet). The profile of $B^{\ast}$ satisfies   $\varphi_{ B^{\ast}}(n)= \vert B\vert^n$.  Thus the profile of a hereditary class of words is either polynomial or at least exponential.

In this paper, we give more information about the binary ordered structures yielding an exponential profile. The case of  ordered structures not necessarily binary is handled in a forthcoming paper \cite{oudrar-pouzet}.

We say that ordered structures of the form $\mathcal R:= (V, \leq, \rho_1, \dots, \rho_k)$ have \emph{type} $k$. As a consequence of  Ramsey's theorem,  we obtain:

\begin{theorem}\label{thm:graph-ordonne}
 The collection $\mathfrak S_k$ of ordered binary structures of type $k$ that do not have a finite monomorphic decomposition has a finite basis,  that is contains  a finite set $\mathfrak A_k$  such that every member of $\mathfrak S_k$ embeds some member of $\mathfrak A_k$. If $k=1$, the subset $\check{\mathfrak A}_1$  of  $\mathfrak A_1$ made of ordered reflexive graphs has  one thousand two hundred and forty two  members and none  embeds into an other one.
\end{theorem}

Members  $\mathcal R$ of $\mathfrak A_k$ have the following common features. First, their domain $V$ is either $\NN\times\{0,1\}$, or $(\NN\times\{0,1\})\cup\{a\}$ for some fixed element $a$. Next, if $\mathcal R$ is such a structure and $u$ is any  one-to-one order-preserving map on $(\NN,\leq)$, then the map $(u,Id)$ on $\NN\times\{0,1\}$ defined by $(u,Id)(x,i)=(u(x),i)$ for $i\in \{0,1\}$ and that fixes $a$ if $a\in V$ preserves $\mathcal R$. They are \emph{almost multichainable} in the sense of \cite{pouzet06}.

Ordered reflexive directed graphs $\mathcal G$ belonging to $\check{\mathfrak A}_1$ depend  on three parameters $p,l,k$. There are integers such that $1\leq p\leq 10$, $1\leq \ell\leq 6$ and the value of $k$ depends upon $p,l$ (see Section \ref{sec:description graphs}).

We show that the profile of every element of $\check{\mathfrak A}_1$ is at least exponential:

\begin{proposition}\label{prop:profil-ordonne}
The profile of a member of $\check{\mathfrak A}_1$ is either given by one of the five following functions:  $\varphi_1(n):=2^n-1$, $\varphi_2(n):=2^n-n$,  $\varphi_3(n):=2^{n-1}$, $\varphi_4(n):=2^{n-1}+1$ and the Fibonacci sequence, or is bounded below by one of them.
\end{proposition}

This paper is organized as follows. Section \ref{sec:basic notions} contains the definitions and basic  notions. Section \ref{sec:monomorphic decomposition}
 contains  a presentation of the notion of monomorphic decomposition. We give the proof  of  Theorem \ref{thm:dichotomy-binary structure} and Corollary \ref{cor:first item} in Subsection \ref{subsection:proofthmdichotomy} and the proof of Proposition \ref{prop:decomp-finie} in Subsection \ref{subsectionproofpropdecfinie}.
In section \ref{sec:proof of theorem graphesordonnes} we give the detailed proof of the first part of Theorem \ref{thm:graph-ordonne} with a description of members of $\mathfrak A_k$. The proof of the second part is given in Section \ref{sec:profile}, with a study of profiles of members of $\check {\mathfrak A}_1$ from which  Proposition \ref{prop:profil-ordonne} follows.
%
%
\smallskip

The results presented in this paper  are included in Chapter 8 of our doctoral thesis \cite{oudrarthese}. They are mentionned in an  abstract  posted on ArXiv
\cite{oudrar-pouzet2014}.
Our results have been presented at the $9$th International Colloquium on Graph Theory and combinatorics (ICGT 2014) held in Grenoble (France), June 30-July 4, 2014, and at the conference-school on Discrete Mathematics and Computer Science (DIMACOS'2015) held in Sidi Bel Abb\`es (Algeria), November 15-19, 2015. We are pleased to thank the organizers of these conferences for their help.

\section{Basic concepts}\label{sec:basic notions}
Our terminology follows Fra\"{\i}ss\'{e} \cite{fraisse}. For properties of profiles we refer to the survey of Pouzet \cite{pouzet06}.
\subsection{Relational structures, embedabbility, hereditary classes and ages}\label{subsec:structures}
Let $m$ be a non-negative integer. A \emph{$m$-ary relation} with \emph{domain} $V$ is a subset $\rho$ of $V^m$. When needed, we identify $\rho$ with its characteristic function$ \chi_{\rho}$ sending every element of $\rho$ on $1$, hence,
$\rho$ becomes a map from $V^{m}$ to  $\{0,1\}$. A \emph{relational structure} is a pair $\mathcal{R}:=(V,(\rho _{i})_{i\in I})$ made of a
set $V$, the \emph{base} or \emph{domain} of $\mathcal R$, also denoted by
 $V(\mathcal R)$ and a family $(\rho _{i})_{i\in I}$ of $n_{i}$-ary relations $\rho _{i}$ on $V$. The family $\mu:=(n_{i})_{i\in I}$ is the \emph{signature} of $\mathcal{R}$. If $I$ is finite, we say that the signature $\mu$ is finite. When all the relations $\rho _{i},~ i\in I$,  are binary, we have a \emph{binary relational structure}, (\emph{binary structure} for short). If one specified relation is a linear order, we have an \emph{ordered relational structure}. A structure which has the form $\mathcal R:=(V,\leq,\rho_1,\dots,\rho_k)$, where $k$ is a non negative integer, $\leq$ is a linear order on the set $V$ and each $\rho_i$ is a binary relation on $V$, is an \emph{ordered binary structure of type} $k$. Basic examples of ordered binary structures are \emph{chains} ($k=0$), \emph{bichains} ($k=1$ and $\rho_1$ is a linear order on $V$) and \emph{ordered directed graphs} ($k=1$); if in this last case $\rho_1$ is an irreflexive and symmetric binary relation on $V$, we just say that the structure is an \emph{ordered graph}.
Let $\mathcal R$ be a relational structure. The \emph{substructure induced} by $\mathcal{R}$ on a subset $A$ of $V$, simply called the
\emph{restriction } of $\mathcal{R}$ to $A$, is the relational structure $\mathcal{R}\restriction_{A}:=(A,(\rho _{i}\restriction _{A})_{i\in I})$.
The notion  of \emph{isomorphism} between relational structures
is defined in the natural way. A
\emph{local isomorphism} of $\mathcal R$ is any isomorphism between two restrictions
of $\mathcal R$.  A relational structure $\mathcal R$ is \emph{embeddable} into a relational structure $\mathcal R'$,  in notation $\mathcal R\leq \mathcal R'$,  if $\mathcal R$ is isomorphic to some restriction of $\mathcal R'$. Embeddability   is a quasi-order on the class of structures having a given signature. We denote by $\Omega_\mu$ the class  of finite relational structures of signature $\mu$;  it is quasi-ordered by embeddability. Most of the time, we consider its members up to isomorphism. The \emph{age} of a relational  structure $\mathcal R$ is the
set $\age (\mathcal R) $ of restrictions of $\mathcal R$ to finite subsets of its domain, these restrictions being considered up to isomorphism.  A class $\mathfrak{C}$ of structures is \textit{hereditary} if it contains every relational structure which can be embedded in some member of $\mathfrak{C}$ (i.e., if $\mathcal R\in\mathfrak C$ and $\mathcal S\leq\mathcal R$ then $\mathcal S\in\mathfrak C$). In order theoretic terms, a  class of finite structures is hereditary iff this  is an initial segment of $\Omega_\mu$. If $\mathfrak B$ is any subset of $\Omega_{\mu}$ the set $Forb (\mathfrak B):= \{\mathcal R \in \Omega_{\mu}: \mathcal B\not \leq \mathcal R \; \text{for all}\;  \mathcal B\in \mathfrak B\} $  is a  hereditary class. A \emph{bound} of a hereditary class $\mathfrak C$ of finite structures is any  element of $\Omega_\mu\setminus \mathfrak C$ which is minimal w.r.t. embeddability. Each hereditary class $\mathfrak C$ of $\Omega_{\mu}$ is determined by its bounds (in fact $\mathfrak C= Forb (\mathfrak B)$ where $\mathfrak B$ is the set of bounds of $\mathfrak C$). Any  age is an \emph{ideal} of $\Omega_\mu$, that is a non empty initial segment of $\Omega_\mu$ which is up-directed. One of the most basic result of the theory of relations asserts that the converse (almost) holds.
\begin{lemma}\label{lem:age}(Fra\"{\i}ss\'e, 1954)
If $\mu$ is finite, every ideal  of $\Omega_\mu$ is the  age of a countable structure.
\end{lemma}\label{lem:fraisse}

\begin{note}
In the sequel, all relational structures we consider will have a finite signature.
\end{note}

\subsection{Invariant structures}

We borrow the following notion and results to \cite{charretton-pouzet}; see
 \cite{Bou-Pouz} for an illustration.

 Let  $C:=(L,\leq )$ be  a chain.
For each integer $n$, let $[C]^{n}$ be the set of $n$-tuples $\overrightarrow{a}:=(a_{1},...,a_{n})\in L^{n}$ such that $a_{1}<...<a_{n}$. This set will be
identified with the set $[L]^n$ of the $n$-element subsets of $L$.

For every local automorphism $h$ of $C$ with domain $D$,
set $h(\overrightarrow{a}):=(h(a_{1}),...,h(a_{n}))$ for every $\overrightarrow{a}\in \lbrack D]^{n}$.
Let $\mathfrak{L}:=\left\langle C,{\mathcal R},\Phi \right\rangle $ be
a triple made of a chain $C$ on $L$,  a relational structure
${\mathcal R}:=(V,(\rho _{i})_{i\in I})$ and a set $\Phi $ of maps, each one being a map
$\psi $ from $[C]^{a(\psi )}$ into $V$, where $a(\psi )$ is an integer, the \emph{arity} of $\psi$.

We say that
$\mathfrak{L}$ is \emph{invariant} if:
\begin{equation}\label{eq:invariance}
\rho _{i}(\psi _{1}(\overrightarrow{\alpha }_{1}),...,\psi _{m_{i}}(%
\overrightarrow{\alpha }_{m_{i}}))=\rho _{i}(\psi _{1}(h(\overrightarrow{%
\alpha }_{1})),...,\psi _{m_{i}}(h(\overrightarrow{\alpha }_{m_{i}})))
\end{equation}
for every $i\in I$ and every local automorphism $h$ of $C$ whose domain contains
$\overrightarrow{\alpha }_{1},...,\overrightarrow{\alpha }_{m_{i}},$
where $m_{i}$ is the arity of $\rho _{i}$, $\psi _{1},...,\psi _{m_{i}}\in\Phi$, $\overrightarrow{\alpha }_{j}\in \lbrack C]^{a(\psi _{j})}$ for $j=1,...,m_{i}.$

Condition (\ref{eq:invariance}) expresses the fact that each $\rho _{i}$ is invariant under the transformation of the $m_{i}$-tuples of $V$, that are induced on $V$,  by the local automorphisms of $C$. For example, if ${\rho}$ is a binary relation and $\Phi =\{\psi\} $ then
\begin{equation*}
\rho(\psi (\overrightarrow{\alpha }),\psi (\overrightarrow{\beta }))=%
\rho(\psi (h(\overrightarrow{\alpha })),\psi (h(\overrightarrow{\beta})))
\end{equation*}%
means that  $\rho (\psi (\overrightarrow{\alpha }),\psi (\overrightarrow{\beta }))$ depends only upon the relative positions of $\overrightarrow{\alpha }$ and $\overrightarrow{\beta }$ on the chain $C$.

If $\mathfrak{L}:=\left\langle C,{\mathcal R},\Phi \right\rangle $ and $L^{\prime }$ is a subset of $L$, set $\Phi _{\upharpoonleft _{L^{\prime
}}}:=\{\psi _{\upharpoonleft _{\lbrack L^{\prime }]^{a(\psi )}}}:\psi \in\Phi \}$ and $\mathfrak{L}_{\upharpoonleft _{L^{\prime }}}:=\left\langle
C_{\upharpoonleft _{L^{\prime }}},{\mathcal R},\Phi _{\upharpoonleft_{L^{\prime }}}\right\rangle $ the restriction of $\mathfrak{L}$ to $L^{\prime }.$

The following result (see \cite{charretton-pouzet}) is a consequence of Ramsey's theorem:

\begin{theorem}\label{thm:ramsey-invariant}
Let $\mathfrak{L}:=\left\langle C,{\mathcal R},\Phi \right\rangle $ be a triple such that the domain $L$ of $C$ is infinite, ${\mathcal R}$
consists of finitely many relations and $\Phi $ is finite. Then there is an infinite subset $L^{\prime }$ of $L$ such that $\mathfrak{L}_{\upharpoonleft
_{L^{\prime }}}$ is invariant.
\end{theorem}

\subsection{Almost multichainable relational structure}\label{subsectionproduct}
A relational structure  $\mathcal R$ of signature $\mu$ is \emph{almost multichainable} if its domain $V$ decomposes into $F \cup (L\times K)$ where $F$ and $K$ are two finite sets  and there is a linear order $\leq$ on $L$ such that:

\begin{equation} \label{eq:local}
\begin{split} \text{For every local isomorphism} \; f\;  \text{of}\; (L, \leq),
\; \text{the map}\;  (f, Id)\; \text{on}\;  L\times K\;\\
  \text{extended}\; \text{by the identity on}\;
F \;  \text{induces a local isomorphism of} \;  \mathcal R.
\end{split}
\end{equation}

  If $F$ is empty and
 $\vert K\vert=1$, Condition (\ref{eq:local}) reduces to say that there is a linear order $\leq$ on $V$ such that   every local isomorphism of  $(V, \leq) $ is a local isomorphism of $\mathcal R$. Relational  structures with this property are said {\it chainable}, a notion invented by Fra\"{\i}ss\'e \cite {fraisse} (the relationship with monomorphy is given in Section \ref{subsection:basic}). \\
 Almost multichainable structures were introduced in \cite{pouzet.tr.1978}, see \cite{pouzet06}. They fall under the frame  of invariant structures. Indeed, let $\mathcal R$ be a relational structure; suppose that its domain $V$  decomposes into $V=  F \cup (L\times K)$ where $F$ and $K$ are two finite sets  and that $\leq$ is a linear order on $L$.  Denote by $\psi_k$  the map from $L$ into $V$ defined by $\psi_k (a):= (a, k)$ if $k\in K$ and $\psi_k(a):=k$ if $k\in F$.  Set $\Phi:= \{\psi_k: k\in K\cup F\}$ and $C:= (L, \leq)$. Then the structure  $\left\langle C,{\mathcal R},\Phi \right\rangle$ is said \emph{invariant} iff  Condition \ref{eq:local} holds.  Hence,  Theorem \ref{thm:ramsey-invariant} allows to extract from any $\mathcal R$ an almost multichainable structure. \subsection{Well-quasi-ordering}

We recall that a poset $P$ is \textit{well-quasi-ordered (w.q.o.)} if $P$ contains no infinite antichain and no infinite descending chain.
We recall the following result ~\cite[3.9 p.~329]{pouzet1981}.
This is  a special instance of a property of posets which is
similar to Nash-William's lemma on minimal bad
sequences~\cite{nash-williams1963}).
\begin{lemma}\label{lem:infinitely bounds}
If a hereditary class $\mathfrak C$ is not w.q.o., it contains an age $\mathfrak D$ which is w.q.o. and has infinitely many bounds.
\end{lemma}

We recall the  following
notion and result (see
Chapter~13 p.~354, of \cite{fraisse}).
A class $\mathfrak C$ of finite structures is \emph{very beautiful} if
for every integer $k$, the collection $\mathfrak C (k)$ of structures
$(S, U_1, \dots, U_k)$, where $S\in \mathfrak C$ and $U_1, \dots, U_k$
are unary relations with the same domain as $S$, has no infinite
antichain w.r.t. embeddability. A crucial property is the following:
\begin{theorem}\label{beautiful}\cite{pouzet72}. A very beautiful age has only finitely many bounds.
\end{theorem}
In the case of binary structures, Theorem \ref {beautiful} has a simple proof.

A straightforward consequence of
Higman's theorem on words (see
\cite{higman52}) is that \emph{the age of an almost multichainable structure is very
  beautiful}. As a consequence:
\begin{lemma}[{\cite[Theorem~4.20]{pouzet06}}]
  \label{lemma.almost_multichainable_age}
  The age of an almost multichainable structure  has only finitely many bounds. \end{lemma}

  The  finiteness of the bounds is a deep result. Frasnay \cite{frasnay65} proved by means of  a clever finite  combinatorial analysis that  chainable relational structures have finitely many bounds. The argument using their very beautiful character  is shorter (but less precise: it gives no estimate on the size of bounds).

 \section{Monomorphic decomposition of a relational structures}\label{sec:monomorphic decomposition}

  We present in this section the notion of monomorphic decomposition of a relational structure. This notion was introduced in \cite{P-T-2013}. The introduction of an equivalence relation makes the presentation simpler and the proofs easier. This equivalence relation appeared in  \cite{oudrar-pouzet2014} and for hypergraphs in \cite{sikaddour-pouzet}. Its study is developped in full in  \cite{oudrarthese} (cf. the third part of the thesis).

\subsection{Monomorphic decomposition: basic properties}\label{subsection:basic}

Let $\mathcal R:= (V, (\rho_i)_{i\in I})$ be a relational structure.
A subset $B$ of $V$  is a \emph{monomorphic part} of $\mathcal R$ if for every non-negative integer $k$ and every pair $A,~A'$ of $k$-element subsets of $V(\mathcal R)$, the induced structures on $A$ and $A'$ are isomorphic whenever $A\setminus B=A'\setminus B$.	
A \emph{monomorphic decomposition} of $\mathcal R$ is a partition $\mathcal  P$ of $V$ into monomorphic parts. Equivalently, it is a partition $\mathcal {P}$ of $V(\mathcal R)$ into blocks such that
for every integer $n$, the induced structures on two $n$-element subsets $A$
and $A^{\prime }$ of $V$ are isomorphic whenever $\vert A\cap
B\vert =\vert A^{\prime}\cap B\vert $ for every $B\in\mathcal {P}$. Each monomorphic part is included into a maximal one. This monomorphic part is unique and called a \emph{monomorphic component}.  The monomorphic components of $\mathcal R$ form a  monomorphic decomposition of $\mathcal R$ of which  every monomorphic decomposition of $\mathcal R$ is a refinement (Proposition 2.12, page 14 of  \cite{P-T-2013}).  
\smallskip

Let $x$ and $y$ be two elements of $V$. Let $F$ be a finite subset of
 $V\setminus \{x,y\}$. We say that $x$ and $y$ are $F$\textit{-equivalent} and we set $x\simeq _{F,\mathcal{R}}y$ if the restrictions of $\mathcal{R}$
to $\{x\}\cup F$ and $\{y\}\cup F$ are isomorphic. Let $k$ be a non-negative
integer, we set $x\simeq _{k,\mathcal{R}}y$ if $x\simeq _{F,\mathcal{R}}y$ for
every $k$-element subset $F$ of $V\setminus \{x,y\}$. We set $x\simeq_{\leq k, \mathcal R}y$ if $x\simeq _{k^{\prime },\mathcal{R}}y$ for every
$k^{\prime }\leq k$. Finally, we say that they are \textit{equivalent} and we set $x\simeq _{\mathcal{R}}y$ if $x\simeq _{k,\mathcal{R}}y$ for every integer $k$. 

The fundamental property, whose proof is easy  (see \cite{sikaddour-pouzet}, \cite{oudrar-pouzet2014}, \cite{oudrarthese}) is this:
\begin{lemma}
The relations $\simeq _{k,\mathcal{R}},~\simeq _{\leq k,\mathcal{R}}$ and
$\simeq _{\mathcal{R}}$ are equivalence relations. Furthermore, the equivalence classes of $\simeq _{\mathcal{R}}$ are the  monomorphic components  of $\mathcal{R}$.  \end{lemma}

Using this  equivalence relation, the proof of the following result, which improves Proposition 2.4 of \cite {P-T-2013},  is straigthforward.

\begin{lemma} \label{compactness}
A relational structure $\mathcal R$ admits a finite monomorphic decomposition if and only if there is some integer $\ell$ such that every member of its age $\age(\mathcal R)$ admits a monomorphic decomposition into at most $\ell$ classes. \end{lemma}

 Let  $\mathcal R$ be a relational structure; if for some non-negative integer $n$ all the restrictions to the $n$-element subsets of its domain are isomorphic,  we say that $\mathcal R$ is $n$-\emph{monomorphic}. We say that  $\mathcal R$ is $(\leq n)$-\emph{monomorphic} if it is $m$-monomorphic for every $m\leq n$ and  that $R$ is \emph{monomorphic} if it is $n$-monomorphic for every integer $n$. Since two finite chains with the same cardinality are isomorphic,  chains are monomorphic. More generally, if there is a linear order $\leq$ on the domain $V$ of a relational structure  $\mathcal R$ such that the local isomorphisms of $C:= (V, \leq)$ are local isomorphisms of $\mathcal R$, then $\mathcal R$ is monomorphic. Structures with this propery are called \emph{chainable}. Every infinite monomorphic relational structure is chainable (Fra\"{\i}ss\'e, 1954, \cite{fraisse}). Every  monomorphic relational structure of finite cardinality large enough (depending on the signature) is chainable (Frasnay 1965 \cite {frasnay65}). It is an easy exercice to show that every binary relational structure on at least $4$ elements which is $(\leq 3)$-monomorphic is chainable. The extension to larger arities is a deep result of
 Frasnay 1965 (\cite{frasnay65}).
He shown the existence of an integer $p$ such that every $p$-monomorphic relational structure $R$ whose maximum of the signature is at most $m$ and domain infinite or sufficiently large is chainable (note that any $p$-monomorphic relational structure on $2p-1$-element is $(\leq p)$-monomorphic  \cite{pouzetMZ}).  The least integer $p$ in the sentence above is the  \emph{monomorphy treshold}, $p(m)$.  Its  value vas given by Frasnay \cite{frasnay90} in 1990: $p(1)=1$, $p(2)= 3$, $p(m)=2m-2$ for $m\geq 3$. For a detailed exposition of this result,  see \cite {fraisse} Chapter 13, notably  p. 378.

Relational structures with a finite monomorphis decomposition have a form close to the chainable ones, as shown by the following result proved in \cite{P-T-2013}, cf. Theorem 1.8 p. 10.

 \begin{theorem}
 \label{thm:charcfinitemonomorph}
 A relational structure $\mathcal R:=(V, (\rho_{i})_{i\in I})$ admits a finite
 monomorphic decomposition if and only if there exists a linear order $\leq$ on $V$ and a finite partition $(V_x)_{x\in X}$ of $V$ into intervals of $(V, \leq)$ such that every local isomorphism of $(V, \leq)$ which preserves each interval is a local isomorphism of $\mathcal R$.
 \end{theorem}

 We may notice that if $\mathcal R$ is ordered then the given order $\leq$ has the property stated in Theorem \ref{thm:charcfinitemonomorph}.

If a structure $\mathcal R$ admits a finite monomorphic decomposition, then it has some restriction with the same age which is almost multichainable.  As a consequence  of Lemma \ref {lemma.almost_multichainable_age} we have:

\begin{corollary}\label{lem:finitely bounds}
The age $\age(\mathcal R)$ of a relational structure $\mathcal R$ with a finite monomorphic decomposition is very beautiful and  has finitely many bounds.
\end{corollary}

We may give an almost self-contained proof. The w.q.o. character of  $\age(\mathcal R)$ is Dickson's Lemma. Indeed, if $V_1, \dots,  V_{\ell}$ is a partition of $V$ into monomorphic blocks, associate to every finite subset $A$ of $V$ the sequence $\vartheta (A):= (\vert A\cap V_i\vert )_{i:=1,\dots,  \ell}$.  If $A$ and $A'$ are two finite subsets and $\vartheta (A) \leq \vartheta (A')$ in the direct product $\NN^{\ell}$ then $\mathcal R_{\restriction A}$ is embeddable  into $\mathcal R_{\restriction A'}$. Since $\NN^{\ell}$ is w.q.o. by Dickson's Lemma,  $\age (\mathcal R)$ is w.q.o. Adding $k$ unary predicates to $\mathcal R$ amounts to replace each $a_i:=\vert A\cap V_i\vert $ by a word of length $a_i$  over an alphabet on $2^{k}$ letters. The w.q.o. character follows from Higman's Theorem on words. The fact that $\mathcal R$ has finitely many bounds follows from Theorem \ref {beautiful}.

\subsection{Proof of Theorem \ref{thm:dichotomy-binary structure} and Corollary \ref{cor:first item}}\label{subsection:proofthmdichotomy}

\subsubsection{Proof of Theorem \ref{thm:dichotomy-binary structure}.}
Suppose that (\ref{itemtheo1}) does not hold. We consider two cases.

\noindent $a)$ $\mathfrak C$ is w.q.o.  by embeddability. \\ In this case, $\mathfrak C$ is a finite union of ideals $\mathfrak C_i$ (this is a special case of a general result about posets due to
Erd\"os-Tarski \cite{erdos-tarski}). According to Lemma \ref{lem:fraisse},  each ideal $\mathfrak C_i$ is the age  of a  structure $\mathcal R_i$. Since $\mathfrak C$ does not satisfy (\ref{itemtheo1}), there is some $\mathcal R_i$ that cannot have a finite monomorphic decomposition. Since $\mathfrak C$ is w.q.o., the set of hereditary subclasses is well founded (Higman \cite{higman52}). Since $\mathfrak C$ contains the age of a structure with no finite monomorphic decomposition, it contains a minimal age with this property.

\noindent $b)$ $\mathfrak C$ is not well-quasi ordered by embeddability.\\
In this case, it contains an infinite antichain. According to Lemma \ref{lem:infinitely bounds},  it contains an age $\mathfrak D$ which is w.q.o. and has infinitely many bounds. According to Lemma \ref{lem:finitely bounds}, no relation having this age can have a finite monomorphic decomposition. Since $\mathfrak D$ is w.q.o., it contains an age minimal with this property.
\hfill $\Box$

\subsubsection{Proof of Corollary \ref{cor:first item}.}

Suppose that $\mathfrak C$ is a finite union of ages $\age(\mathcal R_i)$ of relational structures $\mathcal R_i$ and each $\mathcal R_i$ has a monomorphic decomposition into $\ell_i$ components.  Then, for $\ell:=Max \{\ell_i : i\}$, each member of $\mathfrak C$ has a monomorphic decomposition into at most $\ell$ components.

Now, suppose that there is an integer $\ell$ such that every member of $\mathfrak C$ has a monomorphic decomposition into at most $\ell$ blocks. In this case $\mathfrak C$ cannot satisfy (\ref{itemtheo2}) of Theorem \ref{thm:dichotomy-binary structure}.  Otherwise, $\mathfrak C$ contains an  age $\mathfrak D$ such that no  relational structure $\mathcal R$ having this age has a  finite monomorphic decomposition.  According to Lemma \ref{compactness}, there is no bound on the size of the monomorphic decompositions of the member of $\mathfrak D$, and so the members of $\mathfrak C$. This fact contradicts our  hypothesis. Then,
$\mathfrak C$ satisfy (\ref{itemtheo1}) of Theorem \ref{thm:dichotomy-binary structure}.
\hfill $\Box$
\subsection{ Proof of Proposition \ref{prop:decomp-finie}.}\label{subsectionproofpropdecfinie}

We prove a little more, namely: \\
\emph{The equivalence relation $\simeq_{\mathcal R}$ associated to the ordered structure  $\mathcal R:= (V, \leq, (\rho_i)_{i\in I})$ is the intersection of the equivalence relations $\simeq_{\mathcal R_i}$  associated to $\mathcal R_i:= (V, \leq, \rho_i)$ for $i\in I$}.

Clearly $\simeq_{\mathcal R}$  is included into $\simeq_{\mathcal R_i}$ for each $i\in I$. Conversely, let $x,y\in V$ such that $x\simeq_{\mathcal R_i} y$ for every $i\in I$. Let $F$ be a finite subset of $V\setminus \{x,y\}$. Since $x\simeq_{\mathcal R_i} y$ there is a  local isomorphism  $f_i$ of $\mathcal R_i$ which carries $F \cup \{x\}$ onto $F\cup \{y\}$.  Since $f_i$  is a local isomorphism of $C:= (V, \leq)$ which carries $F \cup \{x\}$ onto $F\cup\{y\}$,  $f_i$  is independent of $i$. Hence,  this is a local isomorphism of $\mathcal R$ proving that  $x\simeq_{\mathcal R}y$. \hfill $\Box$

\section{The case of ordered binary structures}

In this section,  and the next,  we consider ordered binary structures. A crucial property of  ordered  structures is that if two finite substructures  are isomorphic, there is just one isomorphism from one to the other. A repeated use of this property allows us describe the  form  of equivalence classes when these structures are binary.
We start by describing the case of one class. It is immediate to show this:
\begin{lemma}\label{lem:monomorphy}
An ordered binary structure $\mathcal R:= (V, \leq, (\rho_i)_{i\in I})$ is monomorphic iff it is $(\leq 2)$-monomorphic iff  each  relation $\rho_i$ which is reflexive is either a chain which coincide with $\leq$ or its dual, a reflexive clique or an antichain and every relation $\rho_i$ which is irreflexive is either an acyclic tournament which coincide with the strict order $<$  or its dual, a clique or an independent set.
\end{lemma}

In order to describe the equivalence classes in general, we introduce some notation.
 Let $\mathcal R:=(V,\leq, (\rho_i)_{i\in I})$ be an ordered binary structure. Identifying a relation to its characteristic function,  we  set
$d_i(x,y):=(\rho_i(x,y),\rho_i(y,x))$  for all $x,y\in V$ and $i\in I$ and $d(x,y):= (d_i(x,y))_{i\in I}$. We may note that the value of $d(x,y)$ determines the value of $d(y,x)$. In fact, set  $\overline u:= (\beta, \alpha)$ for every $u:=(\alpha, \beta)\in \{0,1\}\times \{0,1\}$ and $\overline {( u_i)_{i\in I}}:= (\overline u_i)_{i\in I}$ for every sequence of members of $\{0,1\}\times \{0,1\}$. Doing so, we have
$d(y,x)= \overline{d(x,y)}$.

 Set $C:=(V,\leq)$. Let $x,y\in V$. If $x<y$ in $V$, we set $[x, y]:= \{z\in V: x\leq z\leq y\}$ and $]x,y [:= \{z\in V: x< z< y\}$. If the order between $x$ and $y$ is not given, we denote by $I_{\leq}(x,y)$ the least interval of $C$ containing $x$ and $y$.
 We recall that a subset $A$ of $V$ is an \emph{interval} of $\mathcal R$ if it is an interval of $C$ and
 \begin{equation}\label{eq:interval}
 d(x,y)=d(x',y)
 \end{equation}
 for every $x,x'\in A$ and $y\not\in A$.

 The following fact is obvious:

 \begin{fact}\label{fact:1-equiv-binaire}
Two elements $x,y$ of $V$ such that $x<y$ are $1$-equivalent if and only if
\begin{equation}\label{equ:first equality}
d(x,z)=d(z,y)  \; \text{for every}\; z\in ]x, y[
\end{equation}
and
\begin{equation}\label{equ:second equality}
d(x,z)=d(y,z)  \; \text{ for every}\; z\notin [x,y].
\end{equation}
\end{fact}

\medskip

\begin{lemma}\label{lem:2equiv}
Let $\mathcal R:=(V,\leq , (\rho _{i})_{i\in I})$ be an ordered binary structure. Two elements $x,y$ of $V$ such that $x<y$ are ($\leq 2$)-equivalent if and only if $[x,y]$ is a $1$-monomorphic  interval of $\mathcal R$ and the restrictions of $\mathcal R$ to any two $2$-element subsets of $[x,y]$ distinct of $\{x,y\}$  are isomorphic.
\end{lemma}
\begin{proof}
Suppose that $x$ and $y$ are ($\leq 2$)-equivalent. We claim that  $[x,y]$ is a $1$-monomorphic  interval of $\mathcal R$. If $[x,y]=\{x,y\}$ then since $x$ and $y$ are $0$-equivalent, $\mathcal R_{\restriction \{x\}}$ and $\mathcal R_{\restriction \{y\}}$ are isomorphic and thus $[x,y]$ is $1$-monomorphic.  Since $x$ and $y$ are $1$-equivalent, the fact that $[x,y]$ is an interval of $\mathcal R$ follows  from $(\ref{equ:second equality})$  of  Fact \ref{fact:1-equiv-binaire}. If $[x,y]\neq\{x,y\}$, let $z\in ]x,y[$ and $a\notin [x,y]$ (if any). Since $x$ and $y$ are $1$-equivalent, $\mathcal R_{\restriction \{x,z\}}$ and $\mathcal R_{\restriction \{z,y\}}$ are isomorphic (and the only isomorphism sends $x$ to $z$ and $z$ to $y$), in particular $\mathcal R_{\restriction \{x\}}$ and $\mathcal R_{\restriction \{z\}}$ are isomorphic, hence $[x,y]$ is $1$-monomorphic. Since $x$ and $y$ are $2$-equivalent, the  restrictions of $\mathcal R$ to $\{x,a,z\}$ and $\{y,a,z\}$ are isomorphic; the only isomorphism  sends $a$ onto $a$, $x$ to $z$ and $z$ to $y$, hence $d(a,z)=d(a,x)=d(a,y)$, proving that  $[x,y]$ is an interval of $\mathcal R$. We look at the   $2$-element subsets  of $[x,y]$ which are distinct from $\{x,y\}$. If $[x,y]=\{x,y\}$ there are none and there is nothing to prove. Otherwise,   $[x,y]\not =\{x,y\}$. Let $z\in ]x,y[$.   Since $x$ and $y$ are $1$-equivalent, $\mathcal R_{\restriction \{x,z\}}$ and $\mathcal R_{\restriction \{z,y\}}$ are isomorphic.
 If there is no other element than $z$ this yields the conclusion of the lemma. If there is an  other element $z'$ we may suppose $z<z'$. Since $x$ and $y$ are $2$-equivalent, there is an isomorphism of  $\mathcal R_{\restriction \{x,z,z'\}}$ onto $\mathcal R_{\restriction \{z,z',y\}}$ and in fact a unique one. It sends $x$ to $z$, $z$ to $z'$ and $z'$ to $y$. It follows that the five restrictions  $\mathcal R_{\restriction \{z,z' \}}$, $\mathcal R_{\restriction \{x,z\}}$, $\mathcal R_{\restriction \{x,z'\}}$, $\mathcal R_{\restriction \{y,z\}}$, $\mathcal R_{\restriction \{y,z'\}}$ are isomorphic. This property  yields immediately the conclusion of the lemma if there is no other element. If there are others, then this property  also says that all the restrictions of $\mathcal R$ to all pairs $\{x,z\}$  for which $z\not =y$ are isomorphic; since the restrictions of $\mathcal R$ to pairs $\{z,z'\}$  are isomorphic to the restrictions of  $\mathcal R$ to such pairs $\{x,z\}$,  the restrictions of $\mathcal R$   to all pairs distinct of  $\{x,y\}$ are isomorphic, as claimed.

Conversely, suppose that $[x,y]$ is a $1$-monomorphic  interval of $\mathcal R$ and the restrictions of $\mathcal R$ to any two $2$-element subsets of $[x,y]$ distinct of $\{x,y\}$  are isomorphic. Let $F$ be a finite subset of $V\setminus \{x, y\}$ and $\varphi$ be the unique order-isomorphism from $F \cup \{x\}$ onto $F\cup \{y\}$. Let  $z, z' \in F$. If $z\in F\setminus[x,y]$ then $\varphi (z)=z$. If $z'\in [x,y]$ then $\varphi(z')\in [x,y]$ and  since $[x,y]$ is a $1$-monomorphic  interval of $\mathcal R$, the map $\varphi$ induces an isomorphism of $\mathcal R_{\restriction \{z,z'\}}$ onto $\mathcal R_{\restriction \{\varphi(z),\varphi (z')\}}$. If $z,z'\in  [x,y]$ then $\varphi(z), \varphi (z')\in [x,y]$. From the $1$-monomorphy of $[x,y]$ and the fact that the $2$-element subsets of $[x,y]$ distinct of $\{x,y\}$  are isomorphic, $\varphi$ induces an isomorphism of $\mathcal R_{\restriction \{z,z'\}}$ onto $\mathcal R_{\restriction \{\varphi(z),\varphi (z')\}}$. It follows that $\varphi$ is an isomorphism of $\mathcal R_{\restriction F\cup \{x\}}$ onto $\mathcal R_{\restriction F\cup \{y\}}$ hence $x\simeq_{F, \mathcal R }y$.  In particular $x$ and $y$ are $(\leq 2)$-equivalent.
\end{proof}
\begin{corollary}\label{lem:formeclasseequiv} Let  $\mathcal R:=(V,\leq , (\rho _{i})_{i\in I})$ be  an ordered binary structure and $A$ be a subset of $V$.
If  $A$ is an interval of $(V, \leq)$ included into some $(\leq 1)$-equivalence class then it is included into a $(\leq 2)$-equivalence class.
\end{corollary}
\begin{proof} We apply Lemma \ref{lem:2equiv}. Let $x, y\in A$ with $x<y$. Since $A$ is an interval of $(V, \leq)$ and all the elements of $A$ are $(\leq 1)$-equivalent,  $[x,y]$ is a $1$-monomorphic  interval of $\mathcal R$ and the restrictions of $\mathcal R$ to any two $2$-element subsets of $[x,y]$ distinct of $\{x,y\}$  are isomorphic, hence by Lemma \ref{lem:2equiv}, $x$ and $y$ are $(\leq 2)$-equivalent.
 \end{proof}

\begin{remark} The fact that two elements $x$ and $y$ are $(\leq 2)$-equivalent does not imply that the interval $[x,y]$ is $2$-monomorphic. Indeed, let $\mathcal R: = (V, \leq , \rho)$ where $\rho$ coincides with $\leq$ except on one ordered pair $(x,y)$ with $x<y$. In such an  example, the set $V$ decomposes into four equivalence classes, namely $\{x,y\}$ and the three open intervals of $(V, \leq)$:  $] \leftarrow x[$, $]x,y[$  and $]x  \rightarrow[$  determined by $x$ and $y$.
\end{remark}

The form of the equivalence classes for an arbitrary ordered binary structure is given in Theorem \ref{lem:formeclasseequiv} below. Before,  we extract the following result from Lemma \ref{lem:2equiv}.

\begin{theorem}\label{thm:bin-ordered}
On  an ordered binary structure $\mathcal R:=(V,\leq , (\rho _{i})_{i\in I})$, the equivalence relations $\simeq _{\leq 2,%
    \mathcal R}$ and $\simeq _{\mathcal R}$ coincide.
\end{theorem}

\begin{proof}
Suppose by contradiction that there are two elements $x,y\in V$ with $x<y$, $x\simeq_{\leq 2,\mathcal R}y$ and $x\not\simeq_{F,\mathcal R}y$ for some finite $F\subseteq V\setminus \{x,y\}$. We may choose $F$ minimal w.r.t inclusion. We claim that $F\subseteq [x,y]$. Indeed, set $F':= [x,y]\cap F$ and $V':=F\cup\{x,y\}$. If $F'\neq F$ then,  by minimality of $F$,  $\mathcal R_{\restriction{F'\cup\{x\}}}$ and
$\mathcal R_{\restriction {F'\cup\{y\}}}$ are isomorphic. Since $[x,y]$ is an interval of $\mathcal R$, $[x,y]\cap V'$ is an interval of $\mathcal R\restriction_{V'}$, hence any isomorphism of $\mathcal R_{\restriction_{F'\cup\{x\}}}$ onto $\mathcal R_{\restriction_{F'\cup\{y\}}}$ extended by the identity on $F\setminus [x,y]$ is an isomorphism of $\mathcal R_{\restriction_{F\cup\{x\}}}$ onto $\mathcal R_{\restriction_{F\cup\{y\}}}$. But then, $x\simeq_{F,\mathcal R}y$. This contradicts our hypothesis and proves our claim. According to Lemma \ref{lem:2equiv}, the restrictions of $\mathcal R$ to $F\cup\{x\}$ and $F\cup\{y\}$ are isomorphic. \end{proof}

\medskip

\begin{theorem} \label{structureclasses}An equivalence class $A$ of an ordered binary structure  $\mathcal R:=(V,\leq , (\rho _{i})_{i\in I})$ is either an interval of $(V,\leq)$ or consists of two distinct elements $x$, $y$ with $x<y$ such that the interval $]x,y[$ forms an other equivalence class.
\end{theorem}

\begin{proof}
 Let $A$ be an equivalence class. Lemma \ref{lem:2equiv} ensures that for every $x<y$ in $A$ the open interval $]x,y[$ is included into a $(\leq 2)$-equivalence class.  From this fact and Theorem \ref{thm:bin-ordered}, it follows that if $\vert A\vert \geq 3$, $A$ is an interval of $(V, \leq)$. Thus, if $A$ is not an interval of $(V, \leq)$ then $A$ is made of two elements $x$, $y$ with $x<y$  and there is some $z\not \in A$ such that $x<z<y$. By Lemma \ref{lem:2equiv}, the open interval $]x,y[$ is included into a $(\leq 2)$-equivalence class distinct from $A$. To conclude, we neeed to prove that $]x,y[$ is equal to this equivalence class. This amounts to prove the following claim:
\begin{claim}\label{sandwich}
Let $x<x'<y<y'$ with $x\simeq _{\leq 2,\mathcal{R}}y$ and $x'\simeq _{\leq 2,\mathcal{R}}y'$, then $x\simeq _{\leq 2,\mathcal{R}}y'$.
\end{claim}
\begin{proof} $a)$ We prove that the conclusion holds if  there  is an fifth  element $z$ in the interval $[x,y']$. For example, suppose $z\in ]x, y[$.  In this case,   from Lemma \ref{lem:2equiv},  we have $x'\simeq _{\leq 2,\mathcal{R}}z$. Now, if $x'<z<y$, we have $z\simeq _{\leq 2,\mathcal{R}} y$  from Lemma \ref{lem:2equiv}; if not, then $x<z<x'$;  since $z \simeq_{\leq 2,\mathcal{R}}y'$,
we  have  $x'\simeq_{\leq 2,\mathcal{R}}y$  from Lemma 4.2. In both cases we have $x'\simeq_{\leq 2, \equiv \mathcal {R}} y$. If $x'<z<y$ the conclusion is similar.

\noindent $b)$ From $a)$ we may suppose that $]x,y'[= \{x',y\}$. In this case, it suffices to prove that $[x,y']$ is a $(\leq 2)$-monomorphic  $\mathcal R$-interval. That is the restrictions of $\mathcal R$  on the six ordered pairs $x,x'$, $x,y$, $x, y'$, $x',y$, $x', y'$, $y,y'$ are isomorphic.
\noindent From the fact that $x\simeq_{\leq 2,\mathcal{R}}y$ the restrictions of $\mathcal R$ to $\{x ,x', y'\}$ and to $\{x', y, y'\}$ are isomorphic
Similarly, from the fact that $x'\simeq_{\leq 2,\mathcal{R}}y'$ the restrictions of $\mathcal R$ to $\{x,x',y\}$ and to $\{x, y, y'\}$ are isomorphic. The first isomorphism yields that the restrictions of $\mathcal R$  on $x,x'$ and $x', y$ are isomorphic, the same for $x,y'$ and $x',y'$ and also for $x'y'$ and $y,y'$. Hence, at least the restrictions of $\mathcal R$ to three of these pairs, namely $x,y'$, $x',y'$, $y,y'$ are isomorphic. The second isomorphism yields that the restrictions of $\mathcal R$  on $x,x'$ and $x,y$ are isomorphic, the same for $x,y$ and $x,y'$ and also for $x'y$ and $y,y'$. Also, three of these pairs, namely $x,x'$, $x,y$, $x,y'$ are isomorphic. The pair $x,y'$ belonging to these two sets, the restrictions of $\mathcal R$ to these five pairs are all isomorphic; since the restrictions of $\mathcal R$ to the remaining pairs $y,y'$ and to $x',y$ are isomorphic, the restrictions of $\mathcal R$ to all these pairs are isomorphic as claimed.  \end{proof}\\
With this, the proof of the theorem is complete.
\end{proof}

\begin{remark} Here is an  example for which none of the   $2$-element equivalence classes is an   interval. Let $V:= \NN\times \{0,1,2\}$. Order $V$ by  $(n,i) \leq (n',i')$ if either $n<n'$ or $n=n'$ and $i\leq i'$; the order type of $(V, \leq)$ is  the lexicographical product $3\cdot \omega$. Let    $\rho$ be the lexicographical product $C_3\cdot \omega$, that is $(n,i) \rho (n',i')$ if either $n<n'$ or $n=n'$ and  $i'=i+1$ (modulo $3$). Let  $\mathcal R:= (V, \leq, \rho)$.   Then,  the equivalence classes are the pairs  $\{(n, 0), (n, 2)\}$ and the singletons $\{(n, 1)\}$ for  $n\in \NN$.

\end{remark}

From Lemma \ref {lem:monomorphy} and Theorem  \ref {structureclasses}, Theorem \ref{thm:charcfinitemonomorph} restricted to binary ordered structures follows. Indeed, let $\mathcal R:=(V,\leq , (\rho _{i})_{i\in I})$ be such a structure. Then, each equivalence class   of $\simeq_{\mathcal R}$ which is not an interval of $(V, \leq)$ has just two elements; replacing each of these classes by  two blocks made of these two elements   will give a partition of $V$ into  monomorphic parts.  On each part, says $A$,  every local isomorphism of $(A, \leq_{\restriction A})$, extended by the identity outside, is a local isomorphism of $\mathcal R$, hence every local isomorphism of $(V, \leq)$ which preserves each interval of this new partition is a local isomorphism of $\mathcal R$. If $\simeq_{\mathcal R}$ has only finitely many classes, the new partition of $V$ has only finitely many blocks and the conclusion of  Theorem \ref{thm:charcfinitemonomorph} holds.

If the number of relations is finite, we have the following separation lemma;
\begin{lemma}\label{lem:separation2}
If an ordered binary structure $\mathcal{R}:=(V,\leq,\rho_1,\dots,\rho_k)$ of type $k$ has infinitely many equivalence classes then one of these
two cases occurs:
\begin{enumerate}
\item\label{item1} There is an infinite subset $A\subseteq V$ such that any two distinct
elements of $A$ are $0$-equivalent but not $1$-equivalent.
\item\label{item2} There are two disjoint infinite subsets $A_{1},A_{2}$ of $V$ such that
any two distinct elements of $A_{i},i\in \{1,2\}$ are $1$-equivalent but
not $2$-equivalent and
 for every $x,y\in A_{i}$, with $i\in \{1,2\}$ and  $x<y$,  we have  $[x,y]\cap A_{j}\neq \emptyset,$
for $i\neq j$.
\end{enumerate}
\end{lemma}

\begin{proof}

\begin{case} $\mathcal R$ has infinitely many classes of $1$-equivalence.
\end{case}Since $\mathcal R$ is made of finitely many relations, $V$ consists of  only finitely many classes of $0$-equivalence. Since $V$ is made of  infinitely many classes of $1$-equivalence,  one class of $0$-equivalence, say $X_{0}$, contains infinitely many classes of $1$-equivalence. Pick   one  element from every class of $1$-equivalence belonging to $X_{0}$ to form a subset $A$ of $V$. The set $A$ satisfies the Assertion (\ref{item1}) of Lemma \ref{lem:separation2}.

\begin{case} $\mathcal R$ has finitely many classes
    of $1$-equivalence.
    \end{case}
    According to Theorem \ref{thm:bin-ordered}, every $(\leq 2)$-equivalence class is an equivalence class, hence $\mathcal R$ has infinitely many $(\leq 2)$-equivalence classes.  Since  $V$  is made of finitely many classes of $1$-equivalence, some $1$-equivalence class, say $X_{1}$,  contains infinitely many $(\leq 2)$-equivalence classes. Pick an element   from each class of $(\leq 2)$- equivalence class included into $X_{1}$.  Let $A$ be the resulting set. Let $a, b\in A$ with $a<b$. Then  the interval  $[a,b]$ cannot be contained in $X_{1}$, otherwise by Item (1) of Lemma  \ref{lem:formeclasseequiv},  $a$ and $b$ would be $(\leq 2)$-equivalent. Hence, there exists $c\in[a,b]$ which belongs to an other class of $1$-equivalence $X'\neq X_{1}$. We can extract from $A$ a sequence $(a_i)_{i\geq 0}$ which is monotonic  w.r.t. $\leq$. With no loss of generality, we may suppose this sequence increasing. According to the above remark, for every $i\geq 0$, there exists $c_i\in [a_i,a_{i+1}]$ with $c_i$ belonging to a class of $1$-equivalence which is different from $X_{1}$. Since the number of $1$-equivalence classes is finite, we can then find an infinite subsequence $(c_i')_{i\geq 0}$ of $(c_i)_{i\geq 0}$ whose elements are in the same class of $1$-equivalence. Let then $(a_i')_{i\geq 0}$ be a subsequence of $(a_i)_{i\geq 0}$ such that $c_i'\in [a_i',a_{i+1}']$. Set $A_1=\{a_i', i\in\mathbb N\}$ and $A_2=\{c_i', i\in\mathbb N\}$.
    The sets $A_1$ and $A_2$ satisfy Assertion (\ref{item2}) of Lemma \ref{lem:separation2}.
\end{proof}

\section{Proof of the first part of Theorem \ref{thm:graph-ordonne}}\label{sec:proof of theorem graphesordonnes}

We give a proof of the first part of Theorem \ref{thm:graph-ordonne}. We prove that the collection $\mathfrak S_k$ of ordered binary structures of type $k$ which do not have a finite monomorphic decomposition, has a  finite basis  $\mathfrak A_k$.
The proof of the second part is given in  section \ref{sec:description graphs}.

The proof of Theorem \ref{thm:graph-ordonne} goes as follow. Let $\mathcal R:=(V,\leq,\rho_1,\dots,\rho_k)$ be an ordered binary structure which has infinitely many equivalence classes. 
According to Lemma \ref{lem:separation2}, we have two cases.

\subsection{Case 1.}\label{case:first}
\noindent $\mathcal R$ satisfies  Assertion $(1)$ of Lemma \ref{lem:separation2}.

In this case,  let  $f:\NN\longrightarrow V$  be a $1$-to-$1$ map  such that $f(\NN)=A$ where $A$ is the set given by Assertion $(1)$ of Lemma \ref{lem:separation2}. Since  $f(n)$ and $f(m)$ are  not $1$-equivalent  for every $n<m$, we may find  some element  $g(n,m)$ witnessing  this fact, meaning that the restrictions of $\mathcal R$ to $\{f(n),g(n,m)\}$ and $\{f(m),g(n,m)\}$ are not isomorphic.

\vspace{1mm}

Let $\Phi:=\{f,g\}$ and $\mathfrak{L}:=\left\langle \omega,\mathcal R,\Phi \right\rangle $, where $\omega$ is the chain made of $\NN$ and the natural order.

Ramsey's theorem in the version of Theorem \ref{thm:ramsey-invariant}, allows us to find an infinite subset $X\subseteq\NN$ such that $\mathfrak L\restriction_X$ is invariant. 
By relabeling $X$ with non-negative integers, we may suppose $X=\NN$ and hence that  $\mathfrak{L}$ is invariant.

\begin{claim}\label{order-graphclaim}
The maps $f$ and $g$ satisfy the following  properties:
\begin{enumerate}
\item\label{ass1} $f(n)\leq f(m)\Leftrightarrow f(n')\leq f(m'),~\forall n<m,~n'<m'$.
\item\label{ass2} $d(f(n), f(m))=d(f(n'), f(m')),~\forall n<m,~n'<m'$.
\item\label{ass3} $g(n,m)\leq g(k,l)\Leftrightarrow g(n',m')\leq g(k',l'),~\forall n<m\leq k<l,~n'<m'\leq k'<l'$.
\item\label{ass4} $d(g(n,m),g(k,l))= d(g(n',m'),g(k',l')),~\forall n<m\leq k<l,~n'<m'\leq k'<l'$.

\item\label{ass5} $g(n,m)\in [f(n),f(m)]\Leftrightarrow g(n',m')\in [f(n'),f(m')],~\forall n<m,~n'<m'$.
\item\label{ass6} $g(n,m)\leq f(k)$ for some integers $n<m<k \Leftrightarrow g(n,m)\leq f(l)$ for every $l>m$.
\item\label{ass7} $d(f(n),g(n,m))= d(f(k),g(k,l)),~\forall n<m,~k<l$.
\item\label{ass8} $d(g(n,m),f(k))= d(g(p,q),f(l)),~\forall n<m<k,~p<q<l$.
\item\label{ass9} If the restrictions $\mathcal R_{\restriction_{\{f(n),g(n,m)\}}}$ and  $\mathcal R_{\restriction_{\{f(k),g(n,m)\}}}$ are isomorphic for some integers $n<m<k$ then $\mathcal R_{\restriction_{\{f(n'),g(n',m')\}}}$ and  $\mathcal R_{\restriction_{\{f(k'),g(n',m')\}}}$ are isomorphic for every $n'<m'<k'$.
\item\label{ass10} $g(n,m)$ and $f(k)$ are different for every distinct integers $n<m$ and $k$.
\item\label{ass11} $g(n,m)\neq g(n',m')$ for every $n<m<n'<m'$.
\end{enumerate}
\end{claim}

\begin{proofclaim}
The nine first items follow from invariance.  To prove Item (\ref{ass10}), suppose that there are integers $n,m,k$ with $n<m$ such that $g(n,m)=f(k)$. By construction of the functions $f$ and $g$, the sets $\{f(n),g(n,m)\}$ and $\{f(m),g(n,m)\}$ have two elements each, hence $k$ cannot be equal to $n$ or to $m$.  According to Item  (\ref{ass1}) and (\ref{ass2}), the restrictions of $\mathcal R$ to $\{f(n),f(k)\}$ and $\{f(m),f(k)\}$ are isomorphic;  thus, if   $g(n,m)=f(k)$ we get that the restrictions of $\mathcal R$ to $\{f(n), g(n,m)\}$ and to $\{f(m),g(n,m))\}$ are isomorphic, contradicting the choice of $g(n,m)$. For the proof of Item (\ref{ass11}),  suppose that there are integers $n<m<n'<m'$ such that $g(n,m)=g(n',m')$. The transformation fixing $n$, $m$ and sending $n'$ onto $m'$ is a local isomorphism of the chain, hence the restrictions of $\mathcal R$ to $\{f(n'),g(n,m)\}$ and  $\{f(m'),g(n,m)\}$ are isomorphic. Replacing $g(n,m)$ by $g(n',m')$, we get that the restrictions to $\{f(n'),g(n',m')\}$ and $\{f(m'),g(n',m')\}$ are isomorphic which is a contradiction with the choice of $g(n',m')$.
\end{proofclaim}

We define a map  $F:\NN\times\{0,1\}\longrightarrow V(\mathcal R)$ and an ordered binary structure of type $k$, $\mathcal R^{(1)}:=(V^{(1)},\leq^{(1)},\rho_1^{(1)},\dots,\rho_k^{(1)})$ with vertex set $V^{(1)}:=\NN\times\{0,1\}$ such that $F$ is an embedding from $\mathcal R^{(1)}$ into $\mathcal R$.

We define first $F$.  We set $F(n,1):=g(2n,2n+1)$ for $n\in\NN$. From Item (\ref{ass9}) of Claim \ref{order-graphclaim}, we have two cases:

Case $(a)$.  $\mathcal R\restriction_{\{f(n),g(n,m)\}}$ and  $\mathcal R\restriction_{\{f(k),g(n,m)\}}$ are isomorphic for some integers $n<m<k$; from invariance, this property holds for all $n<m<k$.

Case $(b)$.  Case $(a)$ does not hold.

 In  Case $(a)$, we set  $F(n,0):=f(2n+1)$ for every $n  \in \NN$. In Case $b$, we  set $F(n,0):=f(2n)$ for every $n\in \NN$.

The map $F$ is $1$-to-$1$ (by construction $f$ is $1$-to-$1$, hence the restriction of $F$ to $\NN\times \{0\}$ is $1$-to-$1$; the restriction of $F$ to $\NN\times \{1\}$  is also $1$-to-$1$ by $(11)$ of Claim \ref{order-graphclaim}; the images of $\NN\times \{0\}$ and  $\NN\times \{1\}$ are disjoint by $(10)$ of Claim \ref{order-graphclaim}.

Since $F$ is $1$-to-$1$, we take for $\mathcal R^{(1)}$ the inverse image of $\mathcal R$.

This amounts to
$$ x\leq^{(1)} y\Leftrightarrow F(x)\leq F(y) \; \text{and} \;  d^{(1)}(x,y)=d(F(x),F(y))$$
\noindent for every $x,y\in \NN\times\{0,1\}$, where $d^{(1)}(x,y):=(d_i^{(1)}(x,y))_{i=1,\dots,k}$ and $d_i^{(1)}(x,y):=(\rho_i^{(1)}(x,y),\rho_i^{(1)}(y,x))$ for every $1\leq i\leq k$.

\begin{claim}\label{claim:infinitelymanyclasses}
 For every $n<m\in \NN$, $(n,0)$ and $(m,0)$ are not $1$-equivalent, hence $\mathcal R^{(1)}$ has infinitely many equivalence classes.
\end{claim}

\begin{proofclaim}  

It suffices to prove that:
\begin{equation}\label{eq2}
\mathcal R^{(1)}\restriction_{\{(n,0),(n,1)\}} \text{ and }\mathcal R^{(1)}\restriction_{\{(m,0),(n,1)\}} \text{ are not isomorphic}
\end{equation}

By definition, $\mathcal R_{\restriction \{f(2n), g(2n, 2n+1)\}}$ and $\mathcal R_{\restriction \{f(2n+1), g(2n, 2n+1)\}}$ are not isomorphic.
In Case $(a)$, $\mathcal R\restriction_{\{f(2n),g(2n,2n+1)\}}$ and  $\mathcal R\restriction_{\{f(2m+1),g(2n,2n+1)\}}$ are isomorphic, hence  $\mathcal R\restriction_{\{f(2n+1),g(2n,2n+1)\}}$ and  $\mathcal R\restriction_{\{f(2m+1),g(2n,2n+1)\}}$ are not isomorphic; since $F(n,0):= f(2n+1), F(m,0)= f(2m+1)$ and $F(n,1)= g(2n, 2n+1)$  that means that  $\mathcal R^{(1)}\restriction_{\{(n,0),(n,1)\}}$   and $  \mathcal R^{(1)}\restriction_{\{(m,0),(n,1)\}}$  are not isomorphic, proving that (\ref{eq2}) holds. In Case $(b)$, $\mathcal R_{\restriction \{f(2n), g(2n, 2n+1)\}}$ and $\mathcal R_{\restriction \{f(2m+1), g(2n, 2n+1)\}}$ are  not  isomorphic. By invariance, $\mathcal R_{\restriction \{f(2m), g(2n, 2n+1)\}}$ and $\mathcal R_{\restriction \{f(2m+1), g(2n, 2n+1)\}}$ are isomorphic, hence\\ $\mathcal R\restriction_{\{f(2n),g(2n,2n+1)\}}$ and  $\mathcal R\restriction_{\{f(2m),g(2n,2n+1)\}}$ are not isomorphic; since $F(n,0):= f(2n), F(m,0)= f(2m)$ and $F(n,1)= g(2n, 2n+1)$  that means that  $\mathcal R^{(1)}\restriction_{\{(n,0),(n,1)\}}$   and $  \mathcal R^{(1)}\restriction_{\{(m,0),(n,1)\}}$  are not isomorphic, proving that (\ref{eq2}) holds.
\end{proofclaim}

\begin{claim} \label{claim:finiteness}
The set $\mathfrak A_k^1$ of  ordered binary structures $\mathcal R^{(1)}$ obtained by this process is finite.\end{claim}
\begin{proofclaim}  According to Claim \ref{order-graphclaim},  $\mathcal R^{(1)}$ is entirely defined by its values on the pairs $((i,j), (i',j'))$ of vertices taken among the four vertices $(0,0)$, $(0,1)$, $(1,0)$ and $(1,1)$;  the values on the other pairs  will be deduced by taking local isomorphisms of $C:=(\NN,\leq^{(1)})$ and using Claim \ref{order-graphclaim}.
\end{proofclaim}	

\medskip

Let us give a hint about the form of the structures which arise
 (the full description in the case of a single binary relation  is given in Section \ref{sec:description graphs}).

 Due to its invariance, the map $d^{(1)}$ is determined by its values on the five ordered pairs
$((0,0), (1,0))$, $((0,1), (1,1))$, $((0,0), (0,1))$, $((0,0),(1,1))$, $((0,1), (1,0))$. The only requirement for  those pairs, due to Condition  \eqref{eq2},  is that   \begin{equation}
d^{(1)}((0,0),(0,1))\neq d^{(1)}((0,1),(1,0)). \label{eq:3}
\end{equation}

On each ordered pair, $d^{(1)}$ can take $4^{k}$ values. On five pairs, this gives $4^{5k}$ possibilities, but $4^{4k}$ are forbidden (those for which $d^{(1)}$ takes the same values on the pairs $((0,0),(0,1))$ and $((0,1),(1,0))$). This gives $4^{4k} .(4^k -1)$ possibilities.
In fact, some of the resulting structures embed into some others. Thus the number of non equimorphic structures is a bit less. 

 \subsection{Case 2.}
$\mathcal R$ satisfies the Assertion $(2)$ of Lemma \ref{lem:separation2}.

Since $A_1$ is infinite,  there is an countable  sequence $(x_n)_{n\in \NN}$ of  elements of $A_1$ which is either increasing or decreasing. Let $(y_n)_{n\in \NN}$ such that $y_n\in I_{\leq} (x_n, x_n+1) \cap A_{2}$.
Set $A'_{0}:= \{x_n: n\in \NN\}$, $A'_{1}:= \{y_n: n\in \NN\}$.

 Then $(A'_1\cup A'_2,\leq_{\restriction_{A'_1\cup A'_2}})$ is ordered as $\omega$ or $\omega^*$. Each of the sets $A'_1$ and $A'_2$ is contained into  a class of $1$-equivalence.  Conditions (\ref{equ:first equality}) and  (\ref{equ:second equality})  of Fact \ref{fact:1-equiv-binaire} are satisfied for every $x,y\in A_i$ and $z\in A_j$, $j\neq i$. We have then two situations:

\vspace{1mm}

\textbf{First situation.} There exists $x_0\in V\setminus (A'_1\cup A'_2)$ which witnesses the fact that $A'_1$ and $A'_2$ are contained in two different classes of $1$-equivalence. As $A'_1\cup A'_2$ is ordered as $\omega$ or $\omega^{\star}$, we may suppose that we have either $x_0\leq a$ or $x_0\geq a$, for every $a\in A'_1\cup A'_2$, because otherwise, we can find some cofinite subset of $A'_1\cup A'_2$ for which we have this condition.
\smallskip

    We may then find maps
    $f',g',g'':\NN\longrightarrow V$ such that, $f'(\NN)=A'_1$, $g'(\NN)=A'_2$,  $g''(\NN):=\{x_0\}$,  $g'(n)\in I_{\leq}(f'(n),f'(n+1)),~\forall n\in\NN$ and the restrictions of $\mathcal R$ to $\{ f'(n), g'(n), g''(0)\}$ and $\{ f'(m), g'(n), g''(0)\}$ are not isomorphic for every $n,m\in\NN$ (in fact the restrictions of $\mathcal R$ to $\{ f'(n), g''(0)\}$ and $\{g'(m), g''(0)\}$ are not isomorphic for every $n,m\in\NN$.
\medskip

As in Case (\ref{case:first}),
 we define  a map  $F':\{a\}\cup(\NN\times\{0,1\})\longrightarrow V(\mathcal R)$, with $a\notin\NN\times\{0,1\}$,  such that $F'(a):=g''(0)$, $F'(n,0):=f'(n)$ and $F'(n,1):=g'(n)$.

The map $F'$ being $1$-to-$1$, we may  define   an ordered binary structure of type $k$, $\mathcal R^{(2)}:=(V^{(2)},\leq^{(2)},\rho_1^{(2)},\dots,\rho_k^{(2)})$ with vertex set $V^{(2)}:=\{a\}\cup (\NN\times\{0,1\})$ such that $F'$ is an embedding from $\mathcal R^{(2)}$ into $\mathcal R$.
This amounts to
$$ x\leq^{(2)} y\Leftrightarrow F(x)\leq F(y) \; \text{and} \;  d^{(2)}(x,y)=d(F(x),F(y))$$
\noindent for every $x,y\in \NN\times\{0,1\}$, where $d^{(2)}(x,y):=(d_i^{(2)}(x,y))_{i=1,\dots,k}$ and $d_i^{(2)}(x,y):=(\rho_i^{(2)}(x,y),\rho_i^{(2)}(y,x))$ for every $1\leq i\leq k$.
 
\medskip

By construction, $\mathcal R^{(2)}$ satisfies Condition (\eqref{eq3}) below, hence has infinitely many equivalence classes. 

\begin{equation}\label{eq3}
\forall n<m\in\NN,~\mathcal R^{(2)}\restriction_{\{a,(n,0),(n,1)\}} \text{ and }\mathcal R^{(2)}\restriction_{\{a,(n,1),(m,0)\}} \text{ are not isomorphic.}
\end{equation}

The ordered binary structure $\mathcal R^{(2)}$ satisfies Observation \ref{observation1} stated  below which follows directly from the fact that the elements of $A_i$ for  $i\in\{1,2\}$ are $1$-equivalent.

\begin{observation}\label{observation1}
\begin{enumerate}
\item $a\leq^{(2)} (n,i)\Leftrightarrow a\leq^{(2)} (m,i),~\forall n<m,~ i\in\{0,1\}$.
\item $a\leq^{(2)} (n,0)\Leftrightarrow a\leq^{(2)} (n,1),~\forall n\in\NN$.
\item $d^{(2)}(a,(n,i))=d^{(2)}(a,(m,i)), ~\forall n<m,~ i\in\{0,1\}$.

\item $d^{(2)}((n,i),(m,i))=d^{(2)}((n',i),(m',i)), \forall n<m,~n'<m',~ i\in\{0,1\}$.
\item $d^{(2)}((n,0),(n,1))=d^{(2)}((n,1),(m,0)), ~\forall n<m$.
\item $d^{(2)}((n,0),(n,1))=d^{(2)}((m,0),(m,1)), ~\forall n<m$.

\item $d^{(2)}((n,0),(m,1))=d^{(2)}((m,1),(m+1,0)), ~\forall n<m$.
\end{enumerate}
\end{observation}

With this we obtain a finite subset $\mathfrak B_k^1$ of ordered binary structures with the same vertex set $\NN\times\{0,1\}\cup \{a\}$. 
\medskip

\textbf{Second situation.} There is no vertex $x_0$ as above. Then, since the elements of $A'_i$, $i\in\{1,2\}$, are $1$-equivalent,  we can deduce from Fact \ref{fact:1-equiv-binaire} that 
two vertices $x,y$ of $A'_i$ are separated by two vertices $z,z'$ such that $z\in I_{\leq}(x,y)\cap A'_j$ with $j\neq i$ and $z'\notin I_{\leq}(x,y)$. In this case and by Lemma \ref{lem:separation2}, the relation between two elements of $A'_i$, for at least one $i=1,2$, is different from the relation between two elements $x,y$, with $x\in A'_1$ and $y\in A'_2$.
We can then define two maps $f_1, g_1:\NN\longrightarrow V(\mathcal R)$ such that, $f_1(\NN)=A'_1$, $g_1(\NN)=A'_2$, $g_1(n)\in I_{\leq}(f_1(n),f_1(n+1)),~\forall n\in\NN$.

\smallskip

Set $F'':\NN\times\{0,1\}\longrightarrow V(\mathcal R)$ such that $F''(n,0):=f(n)$ and $F''(n,1):=g(n)$.\\
We can define an ordered binary structure $\mathcal R^{(3)}:=(V^{(3)},\leq^{(3)},\rho_1^{(3)},\dots,\rho_k^{(3)})$ with vertex set $V^{(3)}=\NN\times\{0,1\}$ such that
$$ x\leq^{(3)} y\Leftrightarrow F''(x)\leq F''(y) \; \text{and}\;  d^{(3)}(x,y)=d(F''(x),F''(y))
$$
for every $x,y\in \NN\times\{0,1\}$, where $d^{(3)}(x,y):=(d_i^{(3)}(x,y))_{i=1,\dots,k}$ and $d_i^{(3)}(x,y):=(\rho_i^{(3)}(x,y),\rho_i^{(3)}(y,x))$ for every $1\leq i\leq k$.
 As we said before, by construction of $\mathcal R^{(3)}$, the order $\leq^{(3)}$ is isomorphic to $\omega$ or $\omega^*$ with $(n,1)\in I_{\leq^{(3)}}((n,0),(n+1,0))$ and
for every $n<m$,  we have either $$d^{(3)}((n,0),(m,0))\neq d^{(3)}((n,0),(m,1)),$$ or $$d^{(3)}((n,0),(m,0))\neq d^{(3)}((n,1),(m,1)),$$ or $$d^{(3)}((n,0),(m,1))\neq d^{(3)}((n,1),(m,1)).$$

Then, with the fact that $A'_1$ and $A'_2$ are, each one, included  into  a class of $1$-equivalence, $\mathcal R^{(3)}$ satisfies the following observation.

\begin{observation}\label{observation2}

\begin{enumerate}
\item $d^{(3)}((n,0),(n,1))=d^{(3)}((m,0),(m,1)), ~\forall n<m$.
\item $d^{(3)}((n,0),(n,1))=d^{(3)}((n,1),(m,0)), ~\forall n<m$.
\item $d^{(3)}((n,0),(m,1))=d^{(3)}((m,1),(m+1,0)), ~\forall n<m$.
\item $d^{(3)}((n,i),(m,i))=d^{(3)}((n',i),(m',i)), ~\forall n<m,~n'<m',~~i\in\{0,1\}$.
\end{enumerate}
\end{observation}

It is then clear that $\mathcal R^{(3)}$ 
has infinitely many equivalence classes. By construction, the set
$\mathfrak B_k^2$ of ordered binary structure obtained in this case is finite.

First, we conclude  that $\mathfrak A_k^1\cup\mathfrak B_k^1\cup\mathfrak B_k^2$ is finite.
Hence, the set  $\mathfrak A_k$ of minimal ordered binary structures (w.r.t embeddability) of  $\mathfrak A_k^1\cup\mathfrak B_k^1\cup\mathfrak B_k^2$ is finite. Next, by construction, $\mathfrak A_k$ is a basis. With that,  the proof of the first part of Theorem \ref{thm:graph-ordonne} is complete.

\section{Description of the ordered directed graphs}\label{sec:description graphs}

 We give in this section the proof of the second part of Theorem \ref{thm:graph-ordonne}.

 Let $\mathfrak A_1$ be the set of ordered binary structures $\mathcal R:=(V, \leq, \rho)$ defined in the previous section and let $\check{\mathfrak A}_1$ be the subset made of  the ordered reflexive directed graphs $\mathcal G:=(V,\leq,\rho)$ of the set $\mathfrak A_1$. The members of $\check{\mathfrak A}_1$ are almost multichains on $F\cup(L\times K)$ such that $L:=\NN$, $\vert K\vert=2$ and $\vert F\vert\leq 1$. 
We prove that the set $\check{\mathfrak A}_1$ contains one thousand two hundred and forty two members, 
 such that $\vert\check{\mathfrak A}_1\cap\mathfrak A_1^1\vert=1122$, $\vert\check{\mathfrak A}_1\cap\mathfrak B_1^1\vert=48$ and $\vert\check{\mathfrak A}_1\cap\mathfrak B_1^2\vert=72$.

 According to the nature of these graphs due, in part to the nature of the order $\leq$, we classify these graphs into several classes which we describe below.

\smallskip

Denote by $\mathcal G_{\ell,k}^{(p)}:=(V_{\ell,k}^{(p)},\leq_{\ell,k}^{(p)},\rho_{\ell,k}^{(p)})$ the ordered directed graphs of $\check{\mathfrak A}_1$, where $p, \ell$ and $k$ are non-negative integers such that $1\leq p\leq 10$, $1\leq \ell\leq 6$. The set of vertices $V_{\ell,k}^{(p)}$ is either $\NN\times\{0,1\}$ (if $\mathcal G_{\ell,k}^{(p)}$ is in $\mathfrak A_1^1\cup \mathfrak B_1^2$) or $\NN\times\{0,1\}\cup\{a\}$ (if $\mathcal G_{\ell,k}^{(p)}$ is in $\mathfrak B_1^1$).

\smallskip

The ordered graphs with the same value of $p$ are said of \emph{class} $p$, their restrictions to $A:=\NN\times\{0\}$ are identical and their restrictions to $B:=\NN\times\{1\}$ also. If they have the same value of $\ell$, then the linear orders  $(V,\leq)$ have the same order-type, $\ell$ takes values from $1$ to $6$ if the linear order is isomorphic to respectively $\omega$, $\omega^*$, $\omega+\omega$, $\omega^*+\omega$, $\omega+\omega^*$, $\omega^*+\omega^*$. We do not consider the cases where the linear order $\leq$ is isomorphic to $\underline{2}^*.\omega$, or to $\underline{2}.\omega^*$ because all the ordered directed graphs  which are obtained in this case are isomorphic to some ones for which the order $\leq$ is isomorphic to $\omega$ or $\omega^{\star}$.
 The integer $k$ enumerates the graphs for all values of $p$ and $\ell$. Different classes have not necessarily the same cardinalities.

\smallskip

 For $p=1$ if $\ell=1,2$ we have $1\leq k\leq 18$ and if $3\leq \ell\leq 6$ we have  $1\leq k\leq 15$. For $p=2,3,4$, if $\ell=1,2$ we have $1\leq k\leq 21$ and if $3\leq \ell\leq 6$ we have  $1\leq k\leq 15$. 
 
 For $5\leq p\leq 10$, if $\ell=1,2$ we have  $1\leq k\leq 22$ and if $3\leq \ell\leq 6$ we have  $1\leq k\leq 24$. The total is one thousand two hundred and forty two as claimed.

 \vspace{1mm}

 In each class $p$, when $\ell=1$,  the linear order  $\leq_{\ell,k}^{(p)}$ is isomorphic to $\omega$. In this case we have,
    $(0,0)<_{\ell,k}^{(p)}(0,1)<_{\ell,k}^{(p)}(1,0)$  when the vertex set is $\NN\times\{0,1\}$,  and $a<_{\ell,k}^{(p)}(0,0)<_{\ell,k}^{(p)}(0,1)<_{\ell,k}^{(p)}(1,0)$  when this set is $\NN\times\{0,1\}\cup\{a\}$. The order is reversed when  $\ell=2$.
If $\ell\geq 3$, the vertex set is $\NN\times\{0,1\}$. All the ordered directed graphs given for $\ell\geq 3$ belong to $\mathfrak A_1^1$ and for $p\geq 5$ they all belong to $\mathfrak A_1^1\cup\mathfrak B_1^2$.

\vspace{1mm}

We will give a graphical representations for some classes. All these representations are done on the following six vertices $(0,0)$, $(1,0)$, $(2,0)$, $(0,1)$, $(1,1)$, $(2,1)$ for the graphs of $\mathfrak A_1^1\cup\mathfrak B_1^2$ and on the following seven vertices $a$, $(0,0)$, $(1,0)$, $(2,0)$, $(0,1)$, $(1,1)$, $(2,1)$ for those of $\mathfrak B_1^1$ (the loops are not shown). These representations are given for $\ell=1$ (the linear order $\leq_{\ell,k}^{(p)}$ is isomorphic to $\omega$).

\vspace{1mm}

Recall that a graph $G$ which is isomorphic to its dual is said \emph{self-dual}, and if
 $G$ is a directed graph, the \emph{symmetrized} of $G$ is the graph $G'$ obtained from $G$ by adding every edge $u:=(x,y)$ such that $u^{-1}:=(y,x)$ is an edge of $G$. Thus $G'$ 
 may be considered as an undirected graph.

 Let $\mathcal G:=(V,\leq,\rho)$ be an ordered reflexive directed graph.
 The subset $E$ of $V^2$ such that $(x,y)\in E$ if and only if $\rho(x,y)=1$ is the edge set of $\mathcal G$ and $G:=(V,E)$ is the \emph{directed graph associated to} the ordered directed graph $\mathcal G$. For $x,y\in V$,  set $d(x,y):=(\rho(x,y),\rho(y,x))$.

We are now ready to describe our ordered reflexive directed graphs of $\check{\mathfrak A}_1$ given in Theorem \ref{thm:graph-ordonne}. According to our notations, we have just to describe the associated directed graphs $G_{\ell,k}^{(p)}=(V_{\ell,k}^{(p)},E_{\ell,k}^{(p)})$.

\medskip

 For $n\in\NN$, set $a_n:=((n,0),(n,1))$.

 \textbf{\underline{Class $p=1$:}} The restrictions of $G_{\ell,k}^{(1)}$ to sets $A:=\NN\times\{0\}$ and $B:=\NN\times\{1\}$  are both antichains.

\vspace{1mm}

  \textbf{I)} If $\ell=1,2$ then $1\leq k\leq 18$. 
  
   The graphs $\mathcal G_{\ell,k}^{(1)}$ for $1\leq k\leq 9$ are in $\mathfrak A_1^1$, they are in $\mathfrak B_1^2$ for $10\leq k\leq 12$ and in $\mathfrak B_1^1$ for $13\leq k\leq 18$.

\vspace{1mm}

$\bullet$ For $1\leq k\leq 12$.
   A pair $(x,x')$ of vertices, where $x=(n,i), x'=(n',i')$, is

\begin{itemize}
\item[$\centerdot$] an edge of $G_{\ell,1}^{(1)}$ if $n=n'$ and $i<i'$;
\item[$\centerdot$] an edge of $G_{\ell,2}^{(1)}$ if $(x',x)$ is an edge of $G_{\ell,1}^{(1)}$. Thus $G_{\ell,2}^{(1)}$ is the dual of $G_{\ell,1}^{(1)}$;
\item[$\centerdot$] an edge of $G_{\ell,3}^{(1)}$ if $n=n'$ and $i\neq i'$. The graph $G_{\ell,3}^{(1)}$ is self-dual;
\item[$\centerdot$] an edge of $G_{\ell,4}^{(1)}$ if $n\leq n'$ and $i<i'$;
\item[$\centerdot$] an edge of $G_{\ell,5}^{(1)}$ if $(x',x)$ is an edge of $G_{\ell,4}^{(1)}$. Thus $G_{\ell,5}^{(1)}$ is the dual of $G_{\ell,4}^{(1)}$;
\item[$\centerdot$] an edge of $G_{\ell,6}^{(1)}$ if it is  either an edge of $G_{\ell,4}^{(1)}$ or an edge of $G_{\ell,5}^{(1)}$. Thus  $G_{\ell,6}^{(1)}$ is the symmetrized of $G_{\ell,4}^{(1)}$ (and of $G_{\ell,5}^{(1)}$), it is self-dual;
\item[$\centerdot$] an edge of $G_{\ell,7}^{(1)}$ if  $i<i'$. The graph $G_{\ell,7}^{(1)}$ is equimorphic to its dual. 
 \item[$\centerdot$] an edge of $G_{\ell,8}^{(1)}$ if either ($n\leq n'$ and $i<i'$) or ($n>n'$ and $i<i'$) or ($n<n'$ and $i>i'$). 
 
 \item[$\centerdot$] an edge of $G_{\ell,9}^{(1)}$ if $(x',x)$ is an edge of $G_{\ell,8}^{(1)}$. Thus  $G_{\ell,9}^{(1)}$ is the dual of $G_{\ell,8}^{(1)}$;
\item[$\centerdot$] an edge of $G_{\ell,10}^{(1)}$ if either ($n\leq n'$ and $i<i'$) or ($n<n'$ and $i>i'$);
\item[$\centerdot$] an edge of $G_{\ell,11}^{(1)}$ if $(x',x)$ is an edge of $G_{\ell,10}^{(1)}$. Thus $G_{\ell,11}^{(1)}$ is the dual of $G_{\ell,10}^{(1)}$;
\item[$\centerdot$] an edge of $G_{\ell,12}^{(1)}$ if  $i\neq i'$. The graph $G_{\ell,12}^{(1)}$ is the symmetrized of $G_{\ell,10}^{(1)}$ (and of $G_{\ell,11}^{(1)}$).
\end{itemize}

\vspace{1mm}

Denote by $\underline{2}$ the poset made of $2:=\{0,1\}$ ordered so that $0<1$. The poset $\underline{2}^*$ is its dual. Denote by $K_2$ the reflexive clique on two vertices and by $\Delta_{\NN}$ the antichain with $\NN$ as vertex set. With this notation, $G_{\ell,1}^{(1)}$ is isomorphic to $\underline{2}.\Delta_{\NN}$, the lexicographic product of $\underline{2}$ by $\Delta_{\NN}$ (that is the antichain $\Delta_{\NN}$ where every vertex is replaced by the chain $\underline{2}$). The graph $G_{\ell,2}^{(1)}$ is isomorphic to $\underline{2}^*.\Delta_{\NN}$, the graph $G_{\ell,3}^{(1)}$ is isomorphic to $K_2.\Delta_{\NN}$, the graph $G_{\ell,6}^{(1)}$ is the half complete bipartite graph of Schmerl- Trotter \cite{S-T} and the graph $G_{\ell,7}^{(1)}$ is isomorphic to the ordinal sum $\Delta_{\NN}+\Delta_{\NN}$.

\vspace{1mm}

$\bullet$  For $13\leq k\leq 18$, the vertex set of the graph is $\{a\}\cup A\cup B$. 

A pair $(x,x')$ of vertices is
\begin{itemize}
\item[$\centerdot$] an edge of $G_{\ell,13}^{(1)}$ if $x=a, x'=(n,1)$;
\item[$\centerdot$] an edge of $G_{\ell,14}^{(1)}$ if $(x',x)$ is an edge of $G_{\ell,13}^{(1)}$. Thus $G_{\ell,14}^{(1)}$ is the dual of $G_{\ell,13}^{(1)}$;
\item[$\centerdot$] an edge of $G_{\ell,15}^{(1)}$ if it is  either an edge of $G_{\ell,13}^{(1)}$ or an edge of $G_{\ell,14}^{(1)}$. Thus  $G_{\ell,15}^{(1)}$ is the symmetrized of $G_{\ell,13}^{(1)}$ (and also of $G_{\ell,14}^{(1)}$);
\item[$\centerdot$] an edge of $G_{\ell,16}^{(1)}$ if either $x=a$ and $x'=(n,0)$ or $x=(n,1)$ and $x'=a$; this graph is self-dual;
\item[$\centerdot$] an edge of $G_{\ell,17}^{(1)}$ if it is  either an edge of $G_{\ell,16}^{(1)}$ or an edge of $G_{\ell,13}^{(1)}$;
\item[$\centerdot$] an edge of $G_{\ell,18}^{(1)}$ if $(x',x)$ is an edge of $G_{\ell,17}^{(1)}$. Thus $G_{\ell,18}^{(1)}$ is the dual of $G_{\ell,17}^{(1)}$;
\end{itemize}

\vspace{1mm}

The graphical representations of these graphs are given in \figurename ~\ref{repre:graphe-classe1}.

\vspace{1mm}

\textbf{II)} If $3\leq \ell\leq 6$, we have the same examples for each value of $\ell$ and their number is $15$, according to the linear order $\leq_{\ell,k}^{(1)}$, the elements of $A$ are placed before those of $B$. In these cases, $\mathcal G_{\ell,k}^{(1)}\in\mathfrak A_1^1$ for every $1\leq k\leq 15$.

$\centerdot$ $G_{\ell,k}^{(1)}=G_{1,k}^{(1)}$ for every $1\leq k\leq 6$.

$\centerdot$ $G_{\ell,k}^{(1)}=G_{1,k+1}^{(1)}$ for every $7\leq k\leq 10$.

$\centerdot$ $G_{\ell,11}^{(1)}$ is obtained from $G_{1,7}^{(1)}$ by adding all edges $((n,1),(m,0))$ for $n\geq m$.

$\centerdot$ $G_{\ell,12}^{(1)}$ is obtained from $G_{1,10}^{(1)}$ by adding all edges $((n,1),(m,0))$ for $n\geq m$, the graph $G_{\ell,12}^{(1)}$ is the dual of $G_{\ell,11}^{(1)}$

$\centerdot$ $G_{\ell,13}^{(1)}$ is undirected. Its edge set is $\{\{(n,0),(n',1)\}; n\neq n'\}$.

$\centerdot$ $G_{\ell,14}^{(1)}$ is obtained from $G_{\ell,13}^{(1)}$ by adding all edges $a_n$ for $n\in\NN$.

$\centerdot$ $G_{\ell,15}^{(1)}$ is obtained from $G_{\ell,13}^{(1)}$ by adding all edges $a_n^{-1}$ for $n\in\NN$.

\vspace{2mm}

\begin{figure}[!hbp]
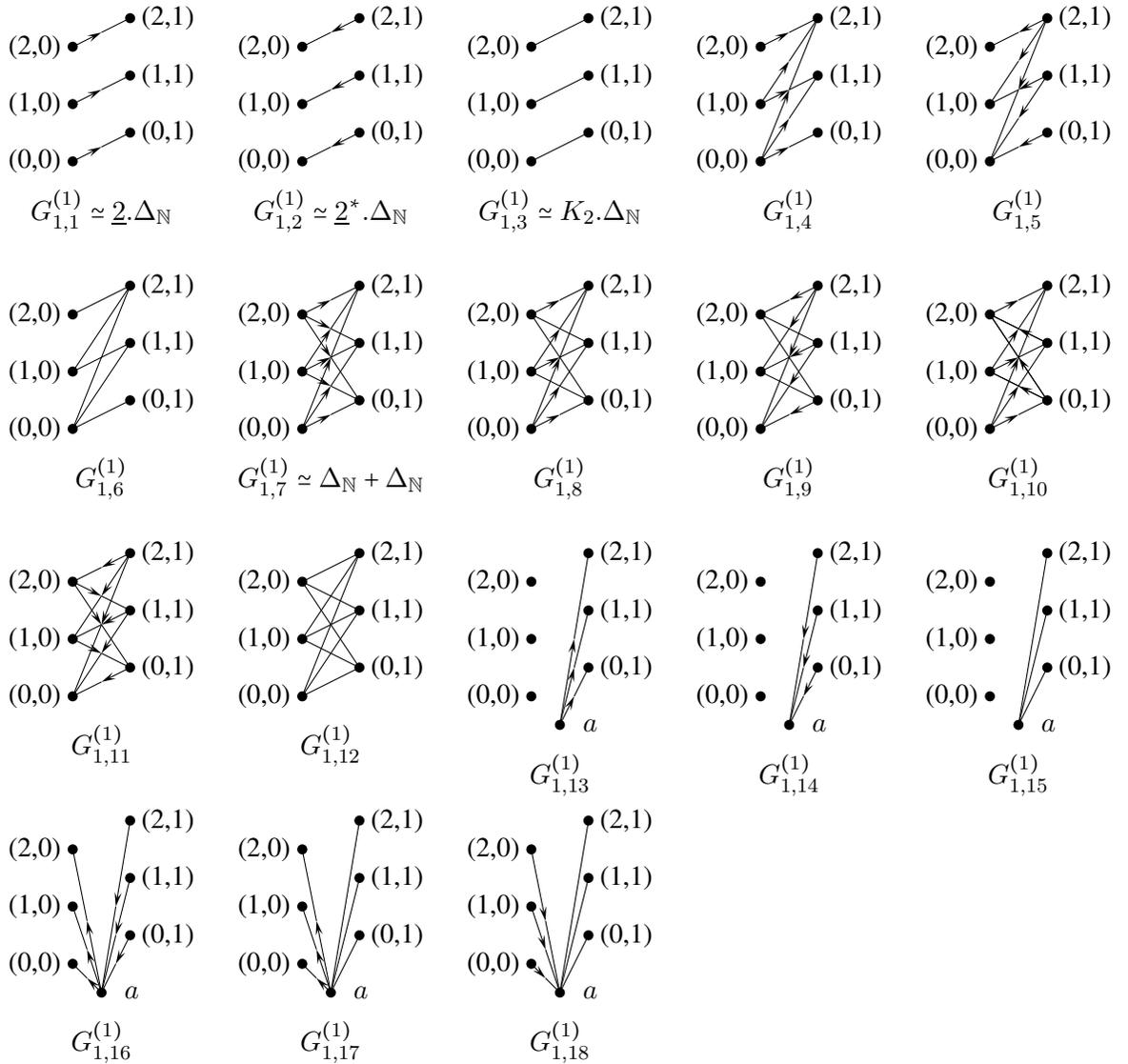

\begin{center}
\small
\begin{tabular}[l]{ccccccccc}
\input{graph1-class1}&&\input{graph2-class1}&&\input{graph3-class1}&&\input{graph4-class1}&&\input{graph5-class1}\\
\input{graph6-class1}&&\input{graph7-class1}&&\input{graph8-class1}&&\input{graph9-class1}&&\input{graph10-class1}\\
\input{graph11-class1}&&\input{graph12-class1}&&\input{graph13-class1}&&\input{graph14-class1}&&\input{graph15-class1}\\
\input{graph16-class1}&&\input{graph17-class1}&&\input{graph18-class1}&&&&\\
&&&&&&&&\\
\end{tabular}

\caption{Minimal graphs of class $p=1$ for $\ell=1$.
The graphs $\mathcal G_{\ell,k}^{(1)}$ are in $\mathfrak A_1^1$ for $1\leq k\leq 9$, in $\mathfrak B_1^2$ for $10\leq k\leq 12$ and in $\mathfrak B_1^1$ for $13\leq k\leq 18$.}
\label{repre:graphe-classe1}
\end{center}
\end{figure}

\vspace{2mm}

 \textbf{\underline{Class $p=2$}:} In this case, the restrictions of $G_{\ell,k}^{(2)}$ to $A$ and $B$ are both chains isomorphic to $\omega$.

\vspace{1mm}

 \textbf{I)} If $\ell=1,2$ we have  $1\leq k\leq 21$, the graphs $\mathcal G_{\ell,k}^{(2)}$ for $1\leq k\leq 12$ are in $\mathfrak A_1^1$, they are in $\mathfrak B_1^1$ for $13\leq k\leq 18$ and in $\mathfrak B_1^2$ for $19\leq k\leq 21$.

\vspace{1mm}

 $\bullet$  For $1\leq k\leq 9$ the graph $G_{\ell,k}^{(2)}$  coincides with $G_{\ell,k}^{(1)}$ on pairs of $A\times B$. Thus $G_{\ell,7}^{(2)}$ is a chain isomorphic to $\omega+\omega$ and the ordered directed graph $\mathcal G_{\ell,7}^{(2)}$ is one of the bichains given in \cite{mont-pou}.

\vspace{1mm}

 $\bullet$ For $10\leq k\leq 12$, the graph $G_{\ell,k}^{(2)}$ 
 coincides with $G_{\ell,10}^{(1)}$ on pairs of $A\times B$ with
 \begin{enumerate}
 \item suppressing edges $a_n, ~n\in\NN$ if $k=10$;  then $G_{\ell,10}^{(2)}$ is isomorphic to $\Delta_2.\omega$, the lexicographical product of the antichain on two vertices $\Delta_2$ with $\omega$.
 \item replacing $a_n$ by $a_n^{-1},~n\in\NN$ if $k=11$; then $G_{\ell,11}^{(2)}$ is isomorphic to $\underline{2}^*.\omega$, the ordered directed graph $\mathcal G_{\ell,11}^{(2)}$ is, in this case, one of the bichains given in \cite{mont-pou}.
 \item adding the edges $a_n^{-1},~n\in\NN$ if $k=12$; then $G_{\ell,12}^{(2)}$ is isomorphic to $K_{2}.\omega$.
 \end{enumerate}

\vspace{1mm}

 $\bullet$ For $13\leq k\leq 18$, the edge set on $(\{a\}\cup A)\times (\{a\}\cup B)$ of the graph $G_{\ell,k}^{(2)}$ is the union of edge sets of  $G_{\ell,10}^{(1)}$ and $G_{\ell,k}^{(1)}$.

 \vspace{1mm}

 $\bullet$ The edge sets of graphs $G_{\ell,19}^{(2)}$ and $G_{\ell,20}^{(2)}$ on $A\times B$ coincide with those of  $G_{\ell,11}^{(1)}$ and $G_{\ell,12}^{(1)}$ respectively.

 \vspace{1mm}

 $\bullet$ The edge set of graph $G_{\ell,21}^{(2)}$ on $A\times B$ is empty.
 Then $G_{\ell,21}^{(2)}$ is isomorphic to $\omega\oplus\omega$, the direct sum of two chains isomorphic to $\omega$.

\vspace{1mm}

The graphical representations of $G_{1,k}^{(2)}$, $1\leq k\leq 21$ are given in \figurename ~\ref{repre:graphe-classe2}.

 \vspace{1mm}

 \textbf{II)} If $3\leq \ell\leq 6$, we have the same examples for each value of $\ell$ and their number is $15$. In these cases, $\mathcal G_{\ell,k}^{(2)}\in\mathfrak A_1^1$ for every $1\leq k\leq 15$.

 $\centerdot$ $G_{\ell,k}^{(2)}=G_{1,k}^{(2)}$ for every $1\leq k\leq 6$ and for every $8\leq k\leq 9$.

 $\centerdot$ $G_{\ell,7}^{(2)}$ is a linear order isomorphic to $\omega$. The ordered directed graph $\mathcal G_{\ell,7}^{(2)}$ in this case is one of the bichains given in \cite{mont-pou}.


$\centerdot$ $G_{\ell,10}^{(2)}=G_{1,19}^{(2)}$.

 $\centerdot$ For $11\leq k\leq 15$, the graph $G_{\ell,k}^{(2)}$ coincides on $A\times B$ with $G_{\ell,k}^{(1)}$.

 \vspace{1mm}

\begin{figure}[!hbp]
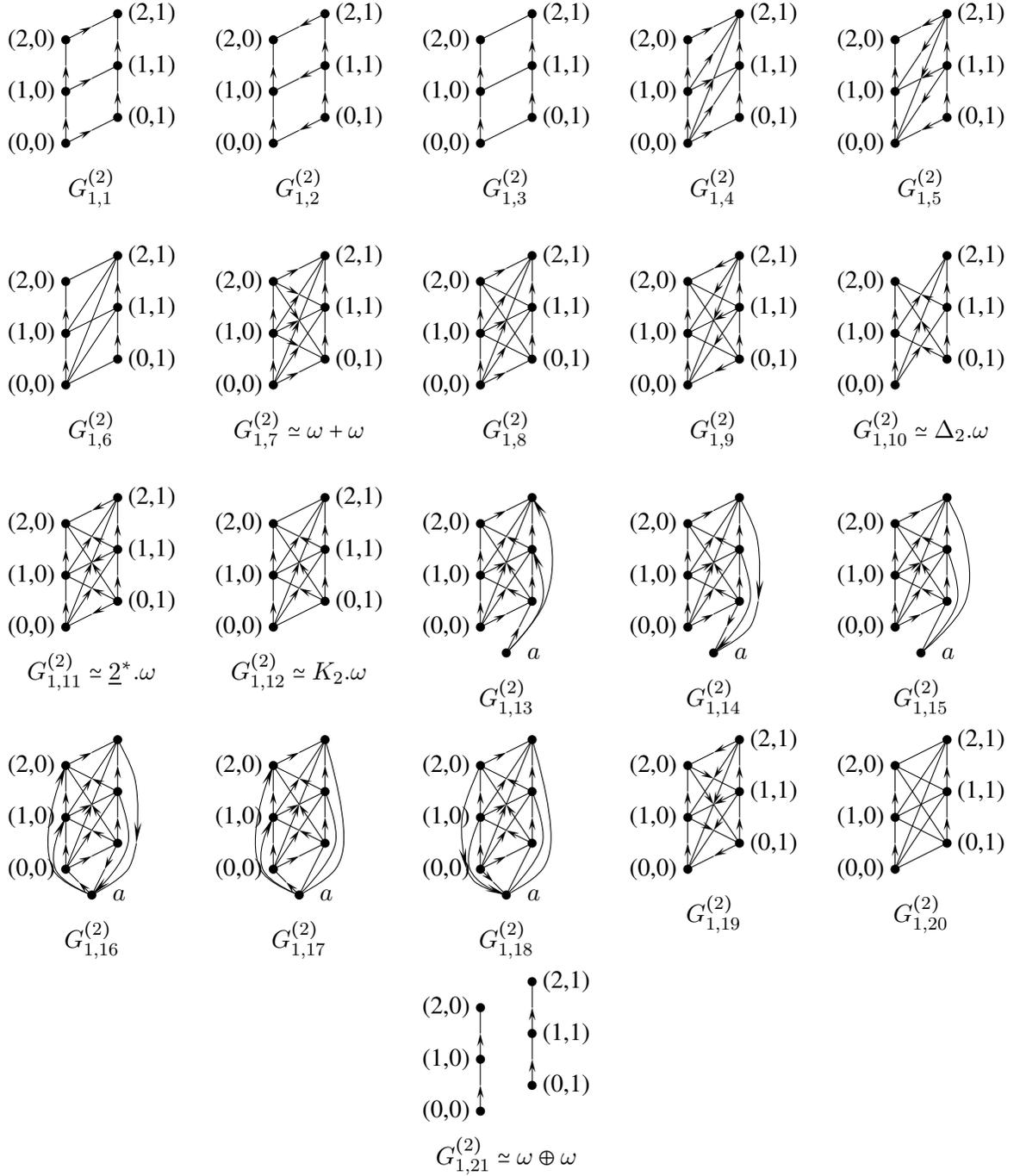

\begin{center}
\small
\begin{tabular}[l]{ccccccccc}
\input{graph1-class2}&&\input{graph2-class2}&&\input{graph3-class2}&&\input{graph4-class2}&&\input{graph5-class2}\\
\input{graph6-class2}&&\input{graph7-class2}&&\input{graph8-class2}&&\input{graph9-class2}&&\input{graph10-class2}\\
\input{graph11-class2}&&\input{graph12-class2}&&\input{graph13-class2}&&\input{graph14-class2}&&\input{graph15-class2}\\
\input{graph16-class2}&&\input{graph17-class2}&&\input{graph18-class2}&&\input{graph19-class2}&&\input{graph20-class2}\\
&&&&\input{graph21-class2}&&&&\\
\end{tabular}
\caption{Minimal graphs of class $p=2$ for $\ell=1$. The graphs $\mathcal G_{\ell,k}^{(2)}$ are in $\mathfrak A_1^1$ for $1\leq k\leq 12$, in $\mathfrak B_1^1$ for $13\leq k\leq 18$ and in $\mathfrak B_1^2$ for $19\leq k\leq 21$.}
\label{repre:graphe-classe2}
\end{center}
\end{figure}

\vspace{2mm}

  \textbf{\underline{Class $p=3$}:} In this case, the restrictions of $G_{\ell,k}^{(3)}$ to $A$ and $B$ are both chains isomorphic to $\omega^*$.

\vspace{1mm}

  \textbf{I)} If $\ell=1,2$ we have $1\leq k\leq 21$, the graphs $\mathcal G_{\ell,k}^{(3)}$ for $1\leq k\leq 12$ are in $\mathfrak A_1^1$, they are in $\mathfrak B_1^1$ for $13\leq k\leq 18$ and in $\mathfrak B_1^2$ for $19\leq k\leq 21$.

\vspace{1mm}

 $\bullet$ For $1\leq k\leq 9$, the graph $G_{\ell,k}^{(3)}$ coincides on $A\times B$ with $G_{\ell,k}^{(1)}$. Then $G_{\ell,7}^{(3)}$ is a chain isomorphic to $\omega^*+\omega^*$ and the ordered directed graph $\mathcal G_{\ell,7}^{(3)}$ is one of the bichains given in \cite{mont-pou}.

\vspace{1mm}

$\bullet$ For $10\leq k\leq 21$, the graph $G_{\ell,k}^{(3)}$ is the dual of $G_{\ell,k}^{(2)}$. Thus the graph $G_{\ell,10}^{(3)}$ is isomorphic to $\Delta_2.\omega^*$, the graph $G_{\ell,11}^{(3)}$ is isomorphic to $\underline{2}.\omega^*$ and hence the ordered directed graph $\mathcal G_{\ell,11}^{(3)}$ is one of the bichains given in \cite{mont-pou}. The graph $G_{\ell,12}^{(3)}$ is isomorphic to $K_{2}.\omega^*$ and the graph  $G_{\ell,21}^{(3)}$ is isomorphic to $\omega^*\oplus\omega^*$.

\vspace{1mm}

 \textbf{II)} If $3\leq \ell\leq 6$, we have the same examples for each value of $\ell$ and their number is $15$. In these cases, $\mathcal G_{\ell,k}^{(3)}\in\mathfrak A_1^1$ for every $1\leq k\leq 15$.

 $\centerdot$ $G_{\ell,k}^{(3)}=G_{1,k}^{(3)}$ for every $1\leq k\leq 6$.

 $\centerdot$ $G_{\ell,7}^{(3)}$ is a linear order isomorphic to $\omega^*$. The corresponding ordered directed graph is one of the bichains given in \cite{mont-pou}.

 $\centerdot$ $G_{\ell,k}^{(3)}=G_{1,k}^{(3)}$ for every $8\leq k\leq 9$.

$\centerdot$ $G_{\ell,10}^{(3)}=G_{1,19}^{(3)}$.

\vspace{1mm}

 $\bullet$ For $11\leq k\leq 15$, the graph $G_{\ell,k}^{(3)}$ coincides on $A\times B$ with $G_{\ell,k}^{(1)}$.

\vspace{2mm}

 \textbf{\underline{Class $p=4$}:} In this case, $A$ and $B$  are both reflexive cliques.

\vspace{1mm}

  \textbf{I)} If $\ell=1,2$ we have $1\leq k\leq 21$, the graphs for $1\leq k\leq 12$ are in $\mathfrak A_1^1$, they are in $\mathfrak B_1^1$ for $13\leq k\leq 18$ and in $\mathfrak B_1^2$ for $19\leq k\leq 21$.

\vspace{1mm}

 $\bullet$ For $1\leq k\leq 9$, the graph $G_{\ell,k}^{(4)}$ coincides with $G_{\ell,k}^{(1)}$ on pairs of $A\times B$. Then $G_{\ell,7}^{(4)}$ is isomorphic to $K_{\NN}+K_{\NN}$, the ordinal sum of two reflexive cliques with the same vertex set $\mathbb N$. 
 
$\centerdot$ The graph $G_{\ell,10}^{(4)}$ is the symmetrized of $G_{\ell,10}^{(2)}$.

$\centerdot$ The graph $G_{\ell,11}^{(4)}$ (respectively $G_{\ell,12}^{(4)}$) is obtained from $G_{\ell,10}^{(4)}$ by adding edges $a_n, n\in\NN$ (respectively $a_n^{-1}, n\in\NN$). The graph $G_{\ell,12}^{(4)}$ is the dual of $G_{\ell,11}^{(4)}$.

\vspace{1mm}

$\bullet$ For $13\leq k\leq 18$, the graph $G_{\ell,k}^{(4)}$ is obtained from $G_{\ell,k}^{(2)}$ by taking its symmetrized on $A\cup B$, the remaining edges (ie, those for which one extremity is $a$) being the same as in $G_{\ell,k}^{(2)}$. 
\vspace{1mm}

$\centerdot$ The graph $G_{\ell,19}^{(4)}$ coincides with $G_{\ell,11}^{(1)}$ on pairs of $A\times B$.

\vspace{1mm}

$\centerdot$ The graph $G_{\ell,20}^{(4)}$ is the dual of $G_{\ell,19}^{(4)}$.

\vspace{1mm}

 $\centerdot$ The graph $G_{\ell,21}^{(4)}$ is the symmetrized of $G_{\ell,21}^{(2)}$, it is isomorphic to $K_{\NN}\oplus K_{\NN}$.

\vspace{1mm}




\textbf{II)} If $3\leq \ell\leq 6$, we have the same examples for each value of $\ell$ and their number is  $15$. In these cases, $\mathcal G_{\ell,k}^{(4)}\in\mathfrak A_1^1$ for every $1\leq k\leq 15$.

 $\centerdot$ $G_{\ell,k}^{(4)}=G_{1,k}^{(4)}$ for every $1\leq k\leq 6$.

 $\centerdot$ $G_{\ell,k}^{(4)}=G_{1,k+1}^{(4)}$ for every $7\leq k\leq 8$.

$\centerdot$ $G_{\ell,9}^{(4)}=G_{1,19}^{(4)}$.

$\centerdot$ $G_{\ell,10}^{(4)}=G_{1,20}^{(4)}$.

\vspace{1mm}

 $\bullet$ For $11\leq k\leq 15$, the graph $G_{\ell,k}^{(4)}$ coincides on $A\times B$ with $G_{\ell,k}^{(1)}$.

\vspace{2mm}

 \textbf{\underline{Class $p=5$}:}  In this case all the graphs have the same vertex set which is $A\cup B$ such that one of the restrictions to $A$ or $B$  is a chain isomorphic to $\omega$, the other being an antichain.

\vspace{1mm}

 \textbf{I)} If $\ell=1,2$ we have $1\leq k\leq 22$, in these cases $\mathcal G_{\ell,k}^{(5)}\in\mathfrak A_1^1$ for $1\leq k\leq 9$ and $13\leq k\leq 21$, the graph $\mathcal G_{\ell,k}^{(5)}\in\mathfrak B_1^2$ for $10\leq k\leq 12$ and $k=22$.

 $\bullet$ For $1\leq k\leq 12$, the graph $G_{\ell,k}^{(5)}$ is such that its restriction to $A$ is a chain, its restriction to $B$ is an antichain and the remaining edges being the same as in $G_{\ell,k}^{(1)}$.

\vspace{1mm}

 $\bullet$ For $13\leq k\leq 21$, the graph $G_{\ell,k}^{(5)}$ is such that its restriction to $B$ is a chain, to $A$ is an antichain, the remaining edges being the same as in $G_{\ell,k-12}^{(1)}$.

\vspace{1mm}

 $\bullet$ The graph $G_{\ell,22}^{(5)}$ is such that its restriction to $A$ is ordered linearly as $\omega$ and its restriction to $B$ is an antichain, there are no other edges. Thus $G_{\ell,22}^{(5)}$ is isomorphic to $\omega\oplus \Delta_{\NN}$.

\vspace{1mm}

The graphical representations of $G_{1,k}^{(5)},~1\leq k\leq 22$ are given in \figurename ~\ref{repre:graphe-classe5}.

\vspace{1mm}

 \textbf{II)} If $3\leq \ell\leq 6$, we have the same examples for each value of $\ell$ and their number is $24$. In these cases, $\mathcal G_{\ell,k}^{(5)}\in\mathfrak A_1^1$ for every $1\leq k\leq 24$.

 $\centerdot$ $G_{\ell,k}^{(5)}=G_{1,k}^{(5)}$ for every $1\leq k\leq 6$.

 $\centerdot$ $G_{\ell,k}^{(5)}=G_{1,k+1}^{(5)}$ for every $7\leq k\leq 10$.

  $\centerdot$ $G_{\ell,k}^{(5)}=G_{1,k+2}^{(5)}$ for every $11\leq k\leq 16$.

   $\centerdot$ $G_{\ell,k}^{(5)}=G_{1,k+3}^{(5)}$ for every $17\leq k\leq 18$.

 $\centerdot$ The graphs $G_{\ell,19}^{(5)}$, $G_{\ell,20}^{(5)}$ and $G_{\ell,21}^{(5)}$ coincide on $A\times B$ respectively with $G_{\ell,13}^{(1)}$, $G_{\ell,14}^{(1)}$ and $G_{\ell,15}^{(1)}$ such that the set $A$ is ordered as $\omega$ and the set $B$ is an antichain.

 $\centerdot$ The graphs $G_{\ell,22}^{(5)}$, $G_{\ell,23}^{(5)}$ and $G_{\ell,24}^{(5)}$ coincide on $A\times B$ respectively with $G_{\ell,13}^{(1)}$, $G_{\ell,14}^{(1)}$ and $G_{\ell,15}^{(1)}$ such that the set $B$ is ordered as $\omega$ and the set $A$ is an antichain.

 \vspace{1mm}

\begin{figure}[!hbp]
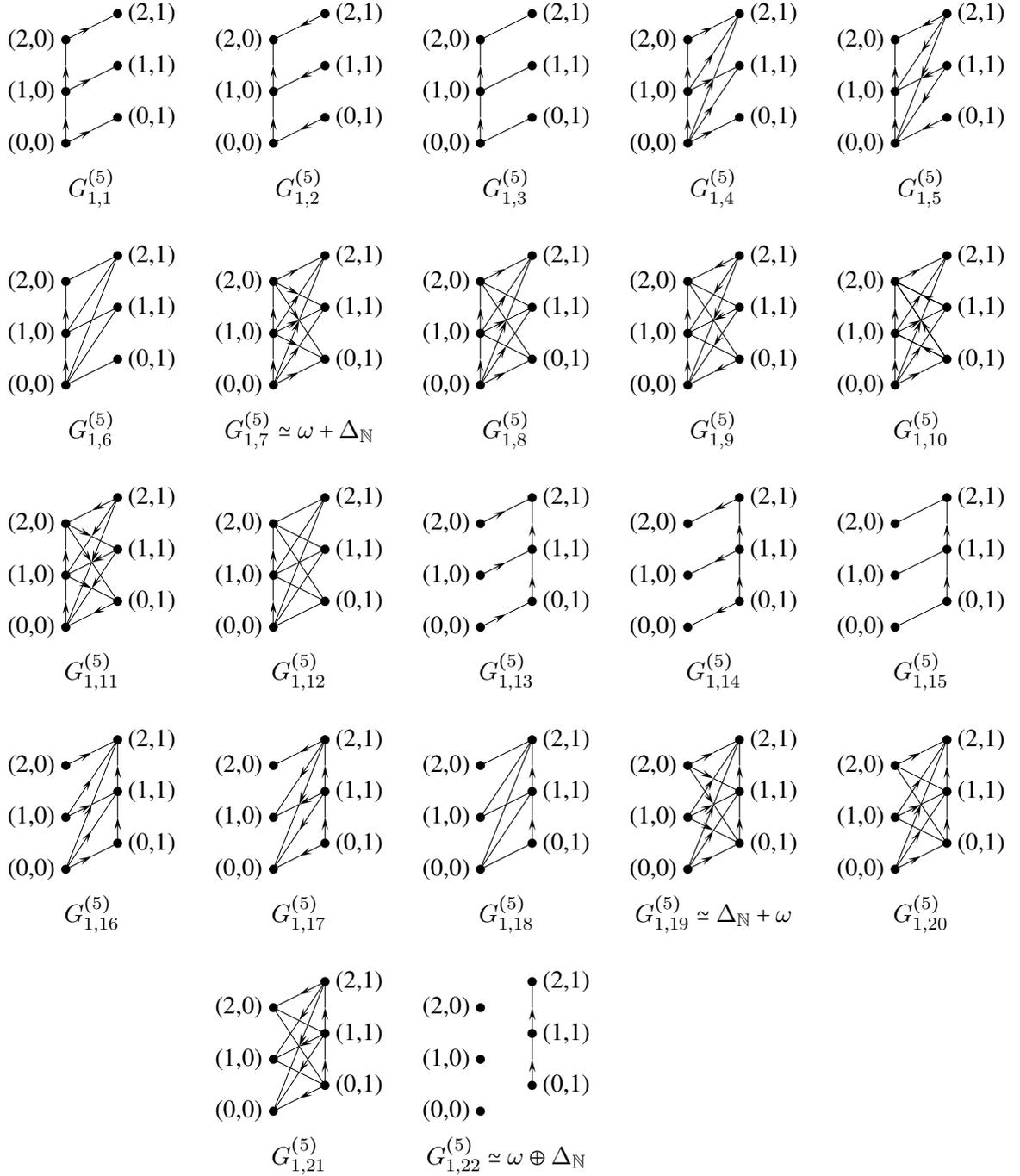

\begin{center}
\small
\begin{tabular}[l]{ccccccccc}
\input{graph1-class5}&&\input{graph2-class5}&&\input{graph3-class5}&&\input{graph4-class5}&&\input{graph5-class5}\\
\input{graph6-class5}&&\input{graph7-class5}&&\input{graph8-class5}&&\input{graph9-class5}&&\input{graph10-class5}\\
\input{graph11-class5}&&\input{graph12-class5}&&\input{graph13-class5}&&\input{graph14-class5}&&\input{graph15-class5}\\
\input{graph16-class5}&&\input{graph17-class5}&&\input{graph18-class5}&&\input{graph19-class5}&&\input{graph20-class5}\\
&&\input{graph21-class5}&&\input{graph22-class5}&&&&\\
\end{tabular}
\caption{The minimal graphs of class $p=5$ for $\ell=1$. The graph $\mathcal G_{1,k}^{(5)}\in\mathfrak A_1^1$ for $1\leq k\leq 9$ and $13\leq k\leq 21$, the graph $\mathcal G_{\ell,k}^{(5)}\in\mathfrak B_1^2$ for $10\leq k\leq 12$ and $k=22$.}
\label{repre:graphe-classe5}
\end{center}
\end{figure}
 \vspace{2mm}

\textbf{\underline{Classes $6\leq p\leq 10$}:} In these cases all the graphs have the same vertex set which is $A\cup B$.
On pairs of $A\times B$, the graphs are obtained in the same way as in case $p=5$, ie, the graph $G_{\ell,k}^{(p)}$ coincides with $G_{\ell,k}^{(5)}$, with the following differences. For $p=6$, the chain on $A$ or $B$ is replaced by a chain isomorphic to $\omega^*$. Hence, every graph in this class is the dual of one graph of the class $p=5$. For $p=7$, the chain on $A$ or $B$ is replaced by a reflexive clique. If $p=8$, the antichain on $A$ or $B$ is replaced by a chain isomorphic to $\omega^*$.
If $p=9$, the antichain on $A$ or $B$ is replaced by a reflexive clique. And in case $p=10$, the chain on $A$ or $B$ is replaced by a chain isomorphic to $\omega^*$ and the antichain is replaced by a reflexive clique.

In these cases, we also obtain bichains among those given in \cite{mont-pou}, they are the ordered directed graphs $\mathcal G_{\ell,k}^{(8)}$ with $\ell\in\{1,2\}$ and $k\in\{7,19\}$.

\medskip

We obtain, among all these ordered graphs, the twenty bichains $\mathcal B:=(V,\leq,\leq')$ of Monteil and Pouzet \cite{mont-pou} and \cite{Bou-Oud-Mon-Pou}.

\medskip

We have the following result.

\begin{lemma}
No graph of the set $\check{\mathfrak A}_1$
embeds into an other one.
\end{lemma}

\begin{proof}
Suppose that there is an embedding $f$ of $\mathcal G_{\ell,k}^{(p)}$ into $\mathcal G_{\ell',k'}^{(p')}$ for some values of $p$, $p'$, $\ell$, $\ell'$, $k$ and $k'$. According to the fact that the restrictions of each graph to $A$ and $B$ on one side and to $A\times B$ on the other do not have the same nature, $A$ is send by $f$ into $A$ or $B$ and $B$ is send to the other. Then we must have $p=p'$. Also, if $\ell\geq 3$, then $\ell'=\ell$ and if $\ell=1,2$, then $(V_{\ell,k}^{(p)}, \leq_{\ell,k}^{(p)})$ is embeddable into $(V_{\ell,k}^{(p)}, \leq_{\ell',k'}^{(p)})$ for some $\ell'\geq 3$ (eg. $\omega$ is embeddable into $\omega+\omega$, $\omega+\omega^*$ and $\omega^*+\omega$) but this embedding is not an embedding from $G_{\ell,k}^{(p)}$ into $G_{\ell',k'}^{(p)}$, it suffices to try to send the four vertices $\{(0,0),(1,0),(0,1),(1,1)\}$ of
$\mathcal G_{\ell,k}^{(p)}$ into $\mathcal G_{\ell',k'}^{(p)}$ preserving the relations. Then, necessarily, $\ell=\ell'$.
Now, if $k\neq k'$, then, if the restrictions of $\mathcal G_{\ell,k}^{(p)}$ to the sets $A$ and $B$ are not isomorphic (it is the case for $p\geq 5$), then the vertices of $A$ are send by $f$ into the set $A$ and those of $B$ are send into $B$. Since these graphs are invariant, it suffices to $f$ to be an embedding from the fourth vertices set $\{(0,0),(1,0),(0,1),(1,1)\}$, if  $\mathcal G_{\ell,k}^{(p)}\in\mathfrak A_1^1\cup\mathfrak B_1^2$ or from the set  $\{a,(0,0),(1,0),(0,1),(1,1)\}$, if $\mathcal G_{\ell,k}^{(p)}\in\mathfrak B_1^1$ such that $f(0,0)$ and $f(1,0)$ are in $A$ and $f(0,1)$ and $f(1,1)$ are in $B$ and fixing $a$ if any. We have no such embedding.
And if the restrictions of $\mathcal G_{\ell,k}^{(p)}$ to $A$ and $B$ are isomorphic (that is the case for $p\leq 4$) then the vertices of $A$ are send by $f$ to vertices of either $A$ or $B$. We can also remark that we can't find an embedding of
$\{(0,0),(1,0),(0,1),(1,1)\}$ or $\{a,(0,0),(1,0),(0,1),(1,1)\}$. Then $k=k'$.
\end{proof}

\smallskip

With this,  the proof of Theorem \ref{thm:graph-ordonne} is complete.

\section{Profiles of members of $\check{\mathfrak A}_1$ and a  proof of Proposition \ref{prop:profil-ordonne}}\label{sec:profile}

For all values of the integers $p,\ell$ and $k$, denote by $\varphi_{\ell,k}^{(p)}$ the profile of the ordered directed graph $\mathcal G_{\ell,k}^{(p)}$. We recall that the graph $G_{\ell,k}^{(p)}$ is the directed graph associated to the ordered graph $\mathcal G_{\ell,k}^{(p)}$.
The proof of  Proposition \ref{prop:profil-ordonne} follows from Lemmas \ref{lemm-profil1} to \ref{lemm-profil6}.

\begin{lemma}\label{lemm-profil1}
 The profile of the ordered graph $\mathcal G_{\ell,k}^{(p)}$ for $1\leq\ell\leq 2$ and ($p=1$ and $1\leq k\leq 3$) or ($2\leq p\leq 4$ and $10\leq k\leq 12$) grows as the Fibonacci sequence.
\end{lemma}

\begin{proof}  There are twenty four ordered graphs to consider. The proof given here take into account all these graphs. The reader can get some help by looking at the graphs $ G_{1,1}^{(1)}$, $ G_{1,2}^{(1)}$, $ G_{1,3}^{(1)}$ represented Figure \ref{repre:graphe-classe1} and the graphs $ G_{1,10}^{(2)}$, $ G_{1,11}^{(2)}$, $ G_{1,3}^{(12)}$ represented Figure \ref{repre:graphe-classe2}.
 Let $\mathcal G_{\ell,k}^{(p)}=(V_{\ell,k}^{(p)},\leq_{\ell,k}^{(p)},\rho_{\ell,k}^{(p)})$ be an ordered graph of our list. As $\ell\in\{1, 2\}$, then $\leq_{\ell,k}^{(p)}$ is ordered as $\omega$ or $\omega^{\star}$.  By invariance, the restrictions of $\mathcal G_{\ell,k}^{(p)}$ to the pairs $\{(n,0),(n,1)\}$, $n\in\NN$,  are all isomorphic. According to the description of the graphs given previously, all other pairs are isomorphic together (see $ G_{1,1}^{(1)}$ for an example).To calculate $\varphi_{\ell,k}^{(p)}(r)$ for $r\in\mathbb N$,  consider $r$ distinct vertices ordered w.r.t $\leq_{\ell,k}^{(p)}$. Then, either this chain ends by a pair of the form $\{(n,0),(n,1)\}$ with $n\in\mathbb N$ and, in this case, the number of non isomorphic subgraphs with $r$ vertices is $\varphi_{\ell,k}^{(p)}(r-2)$, or not.  
And in this latter  case,  the number of such subgraphs of order $r$ is $\varphi_{\ell,k}^{(p)}(r-1)$. We then get:

$$\left\{\begin{array}{l}
\varphi_{\ell,k}^{(p)}(0)=\varphi_{\ell,k}^{(p)}(1)=1.\\
\varphi_{\ell,k}^{(p)}(r)=\varphi_{\ell,k}^{(p)}(r-2)+\varphi_{\ell,k}^{(p)}(r-1)\;\text{ for }r\geq 2.\end{array}
\right.$$
\end{proof}

\begin{lemma}\label{lemme:profilexpo}
The profile of the ordered graph $\mathcal G_{\ell,k}^{(p)}$ for $1\leq\ell\leq 2$ with $(p=1\text{ and } 7\leq k\leq 12)$ or $(p=2\text{ and } k\in\{5,6,9\})$ or $(p=3\text{ and } k\in\{4,6,8\})$ or $(p=4\text{ and }k\in\{4,5,7\})$ is given by $\varphi_{\ell,k}^{(p)}(r)=2^{r}-1,~~r\geq 1$.
\end{lemma}

\begin{proof}
There are thirty ordered  graphs. Nine are represented in Figure  \ref{repre:graphe-classe1} and Figure \ref{repre:graphe-classe2}, namely $ G_{1,k}^{(1)}$, for $7\leq k\leq 12$, and $ G_{1,k}^{(2)}$ for $k\in \{5,6,9\}$. In these cases, we can encode every subgraph with $r$ vertices by a word of length $r$ made of the two letters $\{0,1\}$. Consider $r$ distinct vertices ordered by $\leq_{\ell,k}^{(p)}$. To each vertex  we associate $0$ if  it belongs to $\mathbb N\times\{0\}$ and $1$ if it belongs to $\mathbb N\times\{1\}$, the letters being from left to right. The words in which all letters are identical yield  isomorphic subgraphs, hence  the number of non isomorphic subgraphs with $r$ vertices is
 at most  the number of different words of length $r$ minus $1$. In fact, it is equal.
\end{proof}

\begin{lemma}\label{lemm-profil3}
The profile of the ordered graph $\mathcal G_{\ell,k}^{(p)}$ for $1\leq\ell\leq 2$ with $(p=1, ~k\in\{4,5,6\})$ or $(p=2,~k\in\{4,7,8\})$ or $(p=3,~ k\in\{5,7,9\})$ or $(p=4,~ k\in\{6,8,9\})$ is given by: $\varphi_{\ell,k}^{(p)}(r)=2^{r}-r,~~r\geq 1$.\end{lemma}

\begin{proof} There are twenty four  ordered graphs to consider. Six are represented in Figure  \ref{repre:graphe-classe1} and Figure \ref{repre:graphe-classe2}, namely $ G_{1,k}^{(1)}$, for $k=4,5,6$, and $ G_{1,k}^{(2)}$ for $k=4,7,8$.
In each of these cases, we can also encode, in the same order, every subgraph with $r$ vertices by a binary word of length $r$ as in the proof of lemma \ref{lemme:profilexpo}. For example, for $p=1$, we associate $0$ to each vertex of $\mathbb N\times\{0\}$ and $1$ to each vertex of $\mathbb N\times\{1\}$. If $p=2$ we do the converse. Then all the words of length $r$ of the form $\underset{q}{\underbrace{1\dots1}}\underset{r-q}{\underbrace{0\dots 0}}$ with $q~(0\leq q\leq r)$ are associated to isomorphic subgraphs. Hence, the number of non isomorphic subgraphs with $r$ vertices is at most the number of different words of length $r$ minus $r$. In fact it is equal.\end{proof}

\begin{lemma}\label{lemm-profil2}
The profile of the ordered graph $\mathcal G_{\ell,k}^{(p)}$ for $1\leq\ell\leq 2$ with $(p=1$ and $k\in\{13,14,15\})$ or
$(2\leq p\leq 3$ and $k\in\{13,16,17,19,20,21\})$ or $(p=4$ and $k\in\{15,17,18,19,20,21\})$ is given by $\varphi_{\ell,k}^{(p)}(r)=2^{r-1}$, $r\geq 1$.
\end{lemma}

\begin{proof}
There are forty two ordered graphs. Nine are represented in Figure  \ref{repre:graphe-classe1} and Figure \ref{repre:graphe-classe2}, namely $ G_{1,k}^{(1)}$, for $k=13,14,15$, and $ G_{1,k}^{(2)}$ for $k=13,16,17,19,20,21$. \\
If $k\not\in\{19,20,21\}$, the vertex set for all other graphs cited in the lemma is $\mathbb N\times\{0,1\}\cup\{a\}$. These graphs  have the particularity to be monomorphic on $\mathbb N\times\{0,1\}$ (that is the restrictions to two subsets with the same cardinality are isomorphic). We can encode every subgraph of length $r$ by a word over the alphabet $\{a,0,1\}$. We associate $0$ to each vertex of $\mathbb N\times\{0\}$ and $1$ to each one of $\mathbb N\times\{1\}$ and we add $a$ in the beginning of the word if the subgraph contains the vertex $a$. All words of length $r$ made only with the two letters $0$ and $1$ yield isomorphic subgraphs. Depending on the graph,  these subgraphs are isomorphic to those associated  to words  of length $r$ which begin  by $a$ and whose  remaining letters are identical (identical to  $0$ for  $\mathcal G_{1,13}^{(1)}$ and  identical to $1$ for others as for $\mathcal G_{1,15}^{(4)}$). 
Then, the number of non isomorphic subgraphs of $r$ vertices is equal to the number of different words of length $r$ beginning by $a$. This number is $2^{r-1}$.\\
If $k\in\{ 19, 20, 21\}$, the vertex set of the graphs is $\mathbb N\times\{0,1\}$. These graphs are such that every subgraph with $q$ vertices from $\mathbb N\times\{0\}$ and $r-q$ vertices from $\mathbb N\times\{1\}$ is isomorphic to one of subgraphs with $r-q$ vertices from $\mathbb N\times\{0\}$ and $q$ vertices from $\mathbb N\times\{1\}$. In term of words, the graphs encoded by $a_1a_2\dots a_r$ and by its complement $\overline{a}_1 \overline{a}_2 \dots \overline{a}_r$ where $\overline{0}=1$ and $\overline{1}=0$, are isomorphic. The result follows.
\end{proof}

\begin{lemma}\label{lemm-profil4}
The profile of the ordered graph $\mathcal G_{\ell,k}^{(p)}$ for $1\leq\ell\leq 2$ and $(p=1,~ k\in\{16,17,18\})$ or $(2\leq p\leq 3,~ k\in\{14,15,18\})$ or $(p=4,~ k\in\{13,14,16\})$ is given by $\varphi_{\ell,k}^{(p)}(r)=2^{r-1}+1,~~r\geq 2$.
\end{lemma}

\begin{proof}
There are twenty four ordered graphs. Six are represented in Figure  \ref{repre:graphe-classe1} and Figure \ref{repre:graphe-classe2}, namely $ G_{1,k}^{(1)}$, for $k=16,17,18$, and $ G_{1,k}^{(2)}$ for $k=14,15,18$.
For all the graphs cited in the lemma the vertex set is $\mathbb N\times\{0,1\}\cup\{a\}$. They also have the particularity to be monomorphic on $\mathbb N\times\{0,1\}$. As previously done in the proof of Lemma \ref{lemm-profil2}, we can encode each subgraph on $r$ vertices by a word over the alphabet $\{a,0,1\}$, with the difference that in this case, the subgraphs of order $r$ whose associated words begin by $a$ and are made with only one letter ($0$ or $1$) are not isomorphic to those whose associated word do not begin by $a$. Then, the number of non isomorphic subgraphs of order $r\geq 2$ is equal to the number of words of length $r$ beginning by $a$ plus one. This gives $2^{r-1}+1$.
\end{proof}

\begin{lemma}\label{lemm-profil5}
$\varphi_{\ell,k}^{(p)}(r)\geq 2^{r-1},~~r\geq 2$ for $1\leq\ell\leq 2$ and $(2\leq p\leq 4,~k\in\{1,2,3\})$.
\end{lemma}

\begin{proof}
There are eighteen ordered graphs. Three are represented in  Figure \ref{repre:graphe-classe2}, namely $ G_{1,k}^{(2)}$, for $k=1, 2,3$.
Note that the first values of the profile of these ordered graphs are: \\ $\varphi_{\ell,k}^{(p)}(0)=\varphi_{\ell,k}^{(p)}(1)=1$, $\varphi_{\ell,k}^{(p)}(2)\in\{2,3\}$ and $\varphi_{\ell,k}^{(p)}(3)\in\{6,8\}$ (eg. $\varphi_{1,1}^{(2)}(2)=2$, $\varphi_{1,1}^{(2)}(3)=6$ and $\varphi_{1,2}^{(2)}(2)=3$, $\varphi_{1,2}^{(2)}(3)=8$). Now, to each subgraph with $r$ vertices, ordered according to $\leq_{\ell,k}^{(p)}$, we can associate a word of length $r$ on $\{0, 1\}$ (we associate $0$ if the vertex is in $A$ and $1$ otherwise). If $p=2, k\in\{2,3\}$ or $p=3, k\in\{1,3\}$ or $p=4, k\in\{1,2\}$, two different words are associated to non isomorphic subgraphs, except for the two words of the forme $\underset{q}{\underbrace{0\dots 0}}\underset{r-q}{\underbrace{1\dots 1}}$ and $\underset{q}{\underbrace{1\dots1}}\underset{r-q}{\underbrace{0\dots 0}}$ with $(0\leq q\leq r)$ which are associated to isomorphic subgraphs, but the words containing the factor $01$ are associated to two non isomorphic subgraphs. Indeed, the factor $01$ corresponds to two vertices $(n,0)$ and $(m,1)$ which comes successively according to $\leq_{\ell,k}^{(p)}$, the case $m=n$ leads to a subgraph which is different from those obtained in the case $m>n$. The result follows. \\
If $p=2, k=1$ or $p=3, k=2$ or $p=4, k=3$, all words of the form $\underset{q}{\underbrace{0\dots 0}}\underset{r-q}{\underbrace{1\dots 1}}$ with $(0\leq q\leq r)$ gives the same subgraph if the factor $01$ corresponds to vertices $(n,0)$ and $(m,1)$ with $m>n$. But for each factor $01$ contained in a given word we have two different subgraphs. As there are more than $r$ words with factors $01$ the result follows.
\end{proof}

\begin{lemma}\label{lemm-profil6}
For $\ell\geq 3$ or $p\geq 5$,  the profile of the graph $\mathcal G_{\ell,k}^{(p)}$ is greater or equal to one of the five functions: $\varphi_1(n):=2^n-1$, $\varphi_2(n):=2^n-n$,  $\varphi_3(n):=2^{n-1}$, $\varphi_4(n):=2^{n-1}+1$ and the Fibonacci function.
\end{lemma}

\begin{proof}
There are one thousand and eighty ordered graphs. Twenty two, corresponding to $\ell=1, p=5$ are represented Figure \ref{repre:graphe-classe5}.  All  graphs for $p\geq 5$ are deduced from graphs for $p\leq 4$ with the restrictions to  $\mathbb N\times\{0\}$ and $\mathbb N\times\{1\}$ which are not isomorphic. Hence, the number of subgraphs on $r$ vertices is greater than those obtained in case
$p\leq 4$ where the profiles are given by one of the above five functions. For $\ell\geq 3$, the graphs are obtained from those for which $\ell\leq 2$ by changing the linear order $\leq_{\ell,k}^{(p)}$ (for $3\leq\ell\leq 6$ the linear order is one of the orders $\omega+\omega$, $\omega^{\ast}+\omega$, $\omega+\omega^{\ast}$, $\omega^{\ast}+\omega^{\ast}$), the arguments used in the proofs of previous lemmas remain valid.
\end{proof}

\end{document}